\documentclass{amsart}
\usepackage{amsfonts,amsmath,amssymb,graphicx,graphics,color,
lscape,comment,amsbsy,mathrsfs}

\numberwithin{equation}{section}

\newtheorem{theorem}{Theorem}
\numberwithin{theorem}{section}
\newtheorem{lemma}{Lemma}
\numberwithin{lemma}{section}
\newtheorem{prop}{Proposition}
\numberwithin{prop}{section}
\newtheorem{corol}{Corollary}
\numberwithin{corol}{section}
\newtheorem{remark}{Remark}
\numberwithin{remark}{section}

\numberwithin{defi}{section}

\numberwithin{exe}{section}

\newcommand{\N}{\mathbb{N}}

\newcommand{\UU}{U}

\newcommand{\notthis}[1]{}

\title[An Inversion Formula]{\textbf{An Inversion Formula with Hypergeometric polynomials and Application to Singular Integral Operators}}
\author{R. Nasri(*), A. Simonian(*) and F. Guillemin (**)}
\address{Address:  
(*) Orange Labs, OLN/NMP, Orange Gardens, 
44 avenue de la République, CS 50010, 92326 Chatillon Cedex, France
France (**) Orange Labs Networks Lannion, 2 avenue Pierre Marzin, 22307 Lannion Cedex, Lannion, France}
\email{[ridha.nasri, alain.simonian, fabrice.guillemin]@orange.com}

\begin{document}
%%%%%%%%%%%%%%%%%%%%%%%%%%%%%%%%%%%%%

\date{Version of \today}

%%%%%%%%%%%%%%%%%%%%%%%%%%%%%%%%%%%%%%%%%%%%%%%%%%%
%%%%%%%%%%%%%%%%%%%%%%%%%%%%%%%%%%%%%%%%%%%%%%%%%%%
\begin{abstract}
%%%%%%%%%%%%%%%%%%%%%%%%%%%%%%%%%%%%%%%%%%%%%%%%%%%
%%%%%%%%%%%%%%%%%%%%%%%%%%%%%%%%%%%%%%%%%%%%%%%%%%%

Given complex parameters $x \notin \mathbb{R}^- \cup \{1\}$ and $\nu$, 
$\mathrm{Re}(\nu) < 0$, and the space $\mathscr{H}_0$ of entire  functions in $\mathbb{C}$ vanishing at $0$, we consider the family of integro-differential operators 
$\mathfrak{L} = c_0 \cdot \delta \circ \mathfrak{M}$
with constant $c_0 = \nu(1-\nu)x/(1-x)$, 
$\delta = z \, \mathrm{d}/\mathrm{d}z$ and integral operator 
$\mathfrak{M}$ defined by
$$
\mathfrak{M}f(z) = \int_0^1 e^{- \frac{z}{x}t^{-\nu}(1-(1-x)t)} \, 
f \left ( \frac{z}{x} \, t^{-\nu}(1-t) \right ) \, \frac{\mathrm{d}t}{t}, 
\qquad z \in \mathbb{C},
$$
for all $f \in \mathscr{H}_0$. Inverting $\mathfrak{L}$ or $\mathfrak{M}$ proves equivalent to solve a singular Volterra equation of the first kind. 

The inversion of linear operator $\mathfrak{L}$ on $\mathscr{H}_0$ leads us to derive a new class of linear inversion formulas $T = A(x,\nu) \cdot S \Leftrightarrow S = B(x,\nu) \cdot T$ between sequences $S = (S_n)_{n \in \mathbb{N}^*}$ and $T = (T_n)_{n \in \mathbb{N}^*}$, where the infinite lower-triangular matrix $A(x,\nu)$ and its inverse $B(x,\nu)$ involve Hypergeometric polynomials $F(\cdot)$, namely
$$
	\left\{
	\begin{array}{ll}
	A_{n,k}(x,\nu) = \displaystyle (-1)^k\binom{n}{k}F(k-n,-n\nu;-n;x),
  \\ \\
	B_{n,k}(x,\nu) = \displaystyle (-1)^k\binom{n}{k}F(k-n,k\nu;k;x)
	\end{array} \right.
$$
for $1 \leqslant k \leqslant n$. Functional relations between the ordinary 
(resp. exponential) generating functions of the related sequences $S$ and $T$ are also given. These functional relations finally enable us to derive the integral representation
$$
\mathfrak{L}^{-1}f(z) = \frac{1-x}{2i\pi x} \, e^{z} 
\int_{(0+)}^1 \frac{e^{-xtz}}{t(1-t)} \, 
f \left ( xz \, (-t)^{\nu}(1-t)^{1-\nu} \right ) \, \mathrm{d}t,
\quad z \in \mathbb{C},
$$
for the inverse $\mathfrak{L}^{-1}$ of operator $\mathfrak{L}$ on 
$\mathscr{H}_0$, where the integration contour encircles the point 0 in the complex plane.

%%%%%%%%%%%%%%%%%%%%%%%%%%%%%%%%%%%%%%%%%%%%%%%%%%%
%%%%%%%%%%%%%%%%%%%%%%%%%%%%%%%%%%%%%%%%%%%%%%%%%%%
\end{abstract}
%%%%%%%%%%%%%%%%%%%%%%%%%%%%%%%%%%%%%%%%%%%%%%%%%%%
%%%%%%%%%%%%%%%%%%%%%%%%%%%%%%%%%%%%%%%%%%%%%%%%%%%

\maketitle 

%%%%%%%%%%%%%%%%%%%%%%%%%%
%%%%%%%%%%%%%%%%%%%%%%%%%%

\section{Introduction}

%%%%%%%%%%%%%%%%%%%%%%%%%%
%%%%%%%%%%%%%%%%%%%%%%%%%%

To determine the inverse of an integro-differential operator acting on entire functions in $\mathbb{C}$, we address a new class of linear inversion formulas with coefficients involving Hypergeometric polynomials. After an overview of the state-of-the-art in the associated fields, we then summarize our main contributions.

%%%%%%%%%%%%%%%%%%%%%%%%%%%%%%%%%
\subsection{Motivation}
\label{IM}
%%%%%%%%%%%%%%%%%%%%%%%%%%%%%%%%%
Consider the following problem:

\textit{\textbf{let constants $x \in \; ]0,1[$, 
$\nu < 0$ and the function $\mathfrak{R}$ defined by}}
\begin{equation}
\mathfrak{R}(\zeta) = \left ( 1 - \zeta \right )^{-\nu}
\left ( 1 - (1-x) \, \zeta \right )^{\nu - 1}, 
\qquad \zeta \in [0,1].
\label{DefRR}
\end{equation}
\textit{\textbf{Let $\mathscr{H}_0$ be the linear space of entire functions in $\mathbb{C}$ vanishing at $z = 0$ and define the integro-differential operator 
$\mathfrak{L}:\mathscr{H}_0 \rightarrow \mathscr{H}_0$ by}}
\begin{equation}
\mathfrak{L}f(z) = \int_0^1 
\biggl[ \left ( 1 + z \mathfrak{R}(\zeta) \right ) 
\, f(\zeta \, \mathfrak{R}(\zeta) \cdot z) - c \, z \, 
\mathfrak{R}(\zeta) \,  
f'(\zeta \, \mathfrak{R}(\zeta) \cdot z) \biggr ] 
e^{- \mathfrak{R}(\zeta) \cdot z} \, \mathrm{d}\zeta 
\label{DefTT}
\end{equation}
\textit{\textbf{for all $z \in \mathbb{C}$, where $f'$ denotes the derivative of $f \in \mathscr{H}_0$ and with the constant $c$ in the integrand equal to}}
$$
c = \frac{1-\nu x}{1-x}.
$$

\textit{\textbf{Given $K \in \mathscr{H}_0$, solve the equation}}
\begin{equation}
\mathfrak{L} \, E^*(z) = K(z), 
\qquad z \in \mathbb{C},
\label{EI0}
\end{equation}
\textit{\textbf{for the unknown $E^* \in \mathscr{H}_0$.}}

The operator $\mathfrak{L} = \mathfrak{L}_{x,\nu}$ depends on the set of parameters $U$, $x$ and $\nu$; solving equation (\ref{EI0}) for such parameters is thus equivalent to prove that this operator from $\mathscr{H}_0$ to itself is onto. This inversion problem has been motivated by the resolution of an integral equation arising from Queuing Theory \cite{GQSN18}.

As detailed in the course of this paper, the following \textbf{Properties (I)} and \textbf{(II)} for operator $\mathfrak{L}$ and the associated equation 
(\ref{EI0}) can be respectively  outlined:

\textbf{(I) Operator Factorization: operator $\mathfrak{L}$ can be factored as}
\begin{equation}
\mathfrak{L} =  \frac{x\nu(1-\nu)}{1-x} \cdot \delta \circ \mathfrak{M}
\label{DefTTbis}
\end{equation}
\textbf{where $\delta = z \, \mathrm{d}/\mathrm{d}z$ and $\mathfrak{M}$ is the integral operator defined by}
\begin{equation}
\mathfrak{M}f(z) = 
\int_0^1 e^{- \frac{z}{x}t^{-\nu}(1-(1-x)t)} \, 
f \left ( \frac{z}{x} \, t^{-\nu}(1-t) \right ) \, \frac{\mathrm{d}t}{t}, 
\qquad z \in \mathbb{C},
\label{DefSS}
\end{equation}
\textbf{for all $f \in \mathscr{H}_0$.} Using (\ref{DefTTbis}), solving 
(\ref{EI0}) is therefore equivalent to solve the integral equation 
\begin{equation}
\mathfrak{M}E^* = K_1
\label{DefK1}
\end{equation}
with right-hand side
$$
K_1(z) = \frac{1-x}{\nu(1-\nu)x} \cdot \int_0^z \frac{K(\zeta)}{\zeta} \, \mathrm{d}\zeta, \qquad z \in \mathbb{C},
$$
where $K_1 \in \mathscr{H}_0$ as soon as $K \in \mathscr{H}_0$; specifically, integral equation $\mathfrak{M}E^* = K_1$ can be recast into the Volterra equation 
\begin{equation}
\int_0^{\left ( \frac{\widehat{\tau}}{x} \right ) z} 
\Psi \left ( z, \frac{x\xi}{z}\right ) \, E^*(\xi) \, \mathrm{d}\xi = 
\frac{z}{x} \cdot K_1(z), \qquad z \in \mathbb{C},
\label{Fredh0}
\end{equation}
for some constant $\widehat{\tau}$ and a kernel $\Psi(z,\tau)$ expressed in terms of the two solutions $t = \theta_\pm(\tau) \in [0,1]$ to the implicit equation $t^{-\nu}(1-t) = \tau$. As $\Psi(z,\tau)$ has an integrable singularity of order
\begin{equation}
\Psi(z,\tau) = O \left ( \frac{1}{\sqrt{\widehat{\tau} - \tau}} \right )
\label{Fredh0bis}
\end{equation}
near point $\tau = \widehat{\tau}$, equation (\ref{Fredh0}) with singularity (\ref{Fredh0bis}) therefore belongs to the class of singular Volterra integral equations of the first kind;

\textbf{(II) Reduction to an Infinite Linear System: power series expansions}
\begin{equation}
E^*(z) = \sum_{\ell = 1}^{+\infty} E_\ell \, \frac{z^\ell}{\ell!}, \; \; \; \; 
K(z) = \sum_{b = 1}^{+\infty} (-1)^b K_b \, \frac{z^b}{b!}, 
\qquad z \in \mathbb{C},
\label{defE*}
\end{equation}
\textbf{for a solution $E^* \in \mathscr{H}_0$ and the given 
$K \in \mathscr{H}_0$ reduce the resolution of (\ref{EI0}) to that of the infinite lower-triangular linear system}
\begin{equation}
\forall \, b \in \mathbb{N}^*, \qquad 
\sum_{\ell = 1}^b (-1)^\ell \binom{b}{\ell} \, 
Q_{b,\ell} \, E_{\ell} = K_b,  
\label{T0}
\end{equation}
\textbf{with unknown $E_\ell$, $\ell \in \mathbb{N}^*$, and where the coefficient matrix $Q = (Q_{b,\ell})_{b, \ell \in \mathbb{N}^*}$, on account of the specific function $\mathfrak{R}$ introduced in (\ref{DefRR}), is given by}
\begin{equation}
Q_{b,\ell} = - \frac{\Gamma(b)\Gamma(1-b\nu)}{\Gamma(b-b\nu)} \, 
\frac{x^{1-b}}{1-x} \; F(\ell - b,-b \nu;-b;x), 
\qquad 1 \leqslant \ell \leqslant b.
\label{Q0}
\end{equation}
In (\ref{Q0}), $\Gamma$ is the Euler Gamma function and 
$F(\alpha,\beta;\gamma;\cdot)$ denotes the Gauss Hypergeometric function with complex parameters $\alpha$, $\beta$, $\gamma \notin -\mathbb{N}$. Recall that $F(\alpha,\beta;\gamma;\cdot)$ reduces to a polynomial with degree $-\alpha$ 
(resp. $-\beta$) if $\alpha$ (resp. $\beta$) equals a non positive integer; expression (\ref{Q0}) for coefficient $Q_{b,\ell}$ thus involves a Hypergeometric polynomial with degree $b - \ell$ in both arguments $x$ and $\nu$.

The diagonal coefficients $Q_{b,b}$, $b \geqslant 1$, are non-zero so that lower-triangular system (\ref{T0}) has a unique solution;  equivalently, this proves the uniqueness of the solution $E^* \in \mathscr{H}_0$ to (\ref{EI0}). To make this solution explicit in terms of parameters, write system (\ref{T0}) equivalently as 
\begin{equation}
\forall \, b \in \mathbb{N}^*, \qquad 
\sum_{\ell = 1}^b A_{b,\ell}(x,\nu) \, E_\ell = \widetilde{K}_b,
\label{T0BIS}
\end{equation}
with the reduced right-hand side $(\widetilde{K}_b)$ defined by 
$$
\widetilde{K}_b = - \, 
\frac{\Gamma(b-b\nu)}{\Gamma(b)\Gamma(1-b\nu)} (1-x)x^{b-1}\cdot K_b, 
\qquad b \geqslant 1,
$$ 
and with matrix $A(x,\nu) = (A_{b,\ell}(x,\nu))$ given by
\begin{equation}
A_{b,\ell}(x,\nu) = (-1)^\ell \binom{b}{\ell} F(\ell-b,-b\nu;-b;x), 
\qquad 1 \leqslant \ell \leqslant b.
\label{T0TER}
\end{equation}
As shown in this paper, the linear relation (\ref{T0BIS}) to which initial system (\ref{T0}) has been recast can be explicitly inversed for any right-hand side $(K_b)_{b \in \mathbb{N}^*}$, the inverse matrix 
$B(x,\nu) = A(x,\nu)^{-1}$ involving also Hypergeometric polynomials. This consequently solves system (\ref{T0}) explicitly, hence integral equation 
(\ref{EI0}).

Beside the initial motivation stemming from integral equation 
(\ref{EI0}), the remarkable structure of the inversion scheme 
$B(x,\nu) = A(x,\nu)^{-1}$ obtained in this paper brings a new contribution to the realm of linear inversion formulas, namely infinite lower-triangular matrices with coefficients involving Hypergeometric polynomials.

%%%%%%%%%%%%%%%%%%%%%%%%%%%%%%%%%
\subsection{State-of-the-art}
%%%%%%%%%%%%%%%%%%%%%%%%%%%%%%%%%
\textbf{a)} As mentioned above, equations (\ref{EI0}) or (\ref{DefK1}) can be recast in the form (\ref{Fredh0}) of a singular Volterra integral equation of the first kind. We thus first review known results for this class of integral equations to which our initial problem relates. Given the constant 
$\alpha \in \; ]0,1[$, the typical case of singular equations is given by the classical Abel equation
$$
\int_0^z \frac{E(\xi)}{(z-\xi)^\alpha} \, \mathrm{d}\xi = \kappa(z), 
\qquad z \in [0,r],
$$
on some real interval $[0,r]$, for the unknown function $E$ and some given function $\kappa$ (\cite{BIT95}, Chap.7; \cite{EST00}, Chap.2; \cite{GOR91}, Chap.1). If $\kappa$ is absolutely continuous on $[0,r]$, then Abel equation has a unique solution $E \in \mathrm{L}^1[0,r]$ given by
\begin{align}
E(z) = & \, \frac{\sin(\pi\alpha)}{\pi} \cdot \frac{\mathrm{d}}{\mathrm{d}z} 
\left [ \int_0^z \frac{\kappa(\xi)}{(z-\xi)^{1-\alpha}} \, 
\mathrm{d}\xi \right ]
\nonumber \\
= & \, \frac{\sin(\pi\alpha)}{\pi} 
\left [ \frac{\kappa(0)}{z^{1-\alpha}} + 
\int_0^z \frac{\kappa'(\xi)}{(z-\xi)^{1-\alpha}} \, 
\mathrm{d}\xi \right ], 
\qquad z \in [0,r].
\label{SolAbel}
\end{align}
This solution extends to a complex variable $z \in \mathbb{C}$ pertaining to a neighborhood of point $0$ where function $\kappa$ is assumed to be analytic; the solution $E$ is analytic in a neighborhood of $z = 0$ if condition $\kappa(0) = 0$ holds, that is, $\kappa \in \mathscr{H}_0$. The presence of the derivative $\kappa'$ in 
(\ref{SolAbel}) precludes, however, the continuity of the solution $E$ with respect to the given function $\kappa$ (Abel equation is an ill-posed problem). %This standard framework may nevertheless suggest the existence of an integral representation of the kind (\ref{SolAbel}) for the solution to the presently considered equations (\ref{EI0}) or (\ref{Fredh0}). 

Given a compact $\Omega \subset \mathbb{C}$, general singular operators 
$\mathfrak{J}:E \mapsto \mathfrak{J}E$ given by
$$
\mathfrak{J}E(z) = \int_0^z N(z,\xi)E(\xi) \, \mathrm{d}\xi, 
\qquad z \in \Omega,
$$
where the kernel $N$ verifies
$$
\vert N(z,\xi) \vert \leqslant \frac{M}{\vert z - \xi \vert^\alpha}, 
\qquad z,\xi \in \Omega, \; z \neq \xi,
$$
and some constant $M > 0$, are continuous (even compact) operators on 
$\mathscr{C}^0[\Omega]$ 
(\cite{KRESS14}, Theorem 2.29). No general results are available, however, on the inverse of $\mathfrak{J}$ on some subspace of 
$\mathscr{C}^0[\Omega]$ and, as exemplified above, $\mathfrak{J}^{-1}$ is unbounded if it exists.

As a generalization to the standard formula (\ref{SolAbel}), we will show in this paper how an explicit integral representation for the solutions of either singular equation (\ref{EI0}), (\ref{DefK1}) or (\ref{Fredh0}) can be nevertheless obtained. 

\textbf{b)} We now describe the known classes of linear inversion formulas for the resolution of infinite linear systems. Most of these inversion formulas have been motivated by problems from pure Combinatorics together with the determination of remarkable relations on special functions. 
%Given a complex sequence $(a_j)_{j \in \mathbb{N}}$, it has been early shown \cite{GouldHsu73} that the lower triangular matrices $A$ and $B$ with coefficients
%$$
%A_{n,k} = \frac{1}{(n-k)!} \prod_{j=k}^{n-1}(a_j + k), 
%\quad
%B_{n,k} = \frac{a_k+k}{a_n+n} \cdot \frac{(-1)^{n-k}}{(n-k)!} \prod_{j=k+1}^{n}(a_j + n)
%$$
%for $k \leqslant n$ (with a product over an empty set being set 1) are inverses. This inversion formula actually proves to be a particular case 
Given complex sequences $(a_j)_{j \in \mathbb{Z}}$, $(b_j)_{j \in \mathbb{Z}}$ and $(c_j)_{j \in \mathbb{Z}}$ with $c_j \neq c_k$ for $j \neq k$, it has been shown \cite{Kratten96} that the lower triangular matrices $A$ and $B$ with coefficients
\begin{equation}
A_{n,k} = 
\frac{\displaystyle \prod_{j=k}^{n-1}(a_j + b_j c_k)}
{\displaystyle \prod_{j=k+1}^{n}(c_j - c_k)} , 
\qquad
B_{n,k} = \frac{a_k + b_kc_k}{a_n + b_nc_n} \cdot 
\frac{\displaystyle \prod_{j=k+1}^{n}(a_j + b_j c_n)}
{\displaystyle \prod_{j=k}^{n-1}(c_j - c_n)}
\label{Kratt0}
\end{equation}
for $k \leqslant n$, are inverses. The proof of (\ref{Kratt0}) relies on the existence of linear operators $\mathscr{U}$, $\mathscr{V}$ on the linear space of formal Laurent series such that 
$$
\mathscr{U} f_k(z) = c_k \cdot \mathscr{V} f_k(z), \qquad k \in \mathbb{Z},
$$
where $f_k(z) = \sum_{n \geqslant k} A_{n,k} z^n$; the partial Laurent series 
$g_n(z) = \sum_{k \leqslant n} B_{n,k} z^{-k}$, $n \in \mathbb{Z}$, for the inverse inverse $B = A^{-1}$ can then be expressed in terms of the adjoint operator $\mathscr{V}^*$ of $\mathscr{V}$. A generalization of inverse relation 
(\ref{Kratt0}) to the multi-dimensional case when 
$A = (A_{\mathbf{n},\mathbf{k}})$ with indexes $\mathbf{n}$, 
$\mathbf{k} \in \mathbb{Z}^r$ for some integer $r$ has also been provided in \cite{Schlo97}. As an application, the obtained relations bring summation formulas for multidimensional basic hypergeometric series. 

The lower triangular matrix $A = A(x,\nu)$ introduced in 
(\ref{T0BIS})-(\ref{T0TER}), however, cannot be cast into the specific product form (\ref{Kratt0}) for its inversion: in fact, such a product form for the coefficients of $A(x,\nu)$ should involve the $n-k$ zeros $c_{j,n,k}$, 
$k \leqslant j \leqslant n-1$ of the Hypergeometric polynomial 
$F(k-n,-n\nu;-n;x)$, $k \leqslant n$, in variable $x$; but such zeros depend on all indexes $j$, $n$ and $k$, which precludes the use of a factorization such as (\ref{Kratt0}) where sequences with one index only intervene. In this paper, using functional operations on specific generating series related to its coefficients, we will show how matrix $A(x,\nu)$ can be nevertheless  inverted through a fully explicit procedure.

%%%%%%%%%%%%%%%%%%%%%%%%%%%%%%%%%
\subsection{Paper contribution}
%%%%%%%%%%%%%%%%%%%%%%%%%%%%%%%%%
The main contributions of this paper can be summarized as follows:

$\bullet$ in Section \ref{LTS}, we first establish an inversion criterion for a class of infinite lower-triangular matrices, enabling us to state the inversion formula for the considered class of lower triangular matrices 
(\ref{T0TER}) with Hypergeometric polynomials;

$\bullet$ in Section \ref{GF}, functional relations are obtained for ordinary 
(resp. exponential) generating functions of sequences related by the inversion formulas;

$\bullet$ we end by an application section where we first prove the reduction of integral equation (\ref{EI0}) to the infinite linear system (\ref{T0}) with coefficients depending on Hypergeometric polynomials (\textbf{Reduction Property II} asserted above). Applying the general results of the previous sections, the linear system (\ref{T0}) is fully solved (Section \ref{SecA1}). We next justify the factorization property for the integro-differential operator $\mathfrak{L}$ (\textbf{Factorization Property I} formulated above). The functional relations for exponential generating functions eventually provide us with an integral representation of the inverse $\mathfrak{L}^{-1}$ of operator 
$\mathfrak{L}$ in space $\mathscr{H}_0$ (Section \ref{SecA2}); such a representation reads as an integral in the complex plane with a finite contour encircling the origin 0.

%%%%%%%%%%%%%%%%%%%%%%%%%%%%%%%%%%%%%%%%%%%%%
%%%%%%%%%%%%%%%%%%%%%%%%%%%%%%%%%%%%%%%%%%%%%

\section{Lower-Triangular Systems}
\label{LTS}

%%%%%%%%%%%%%%%%%%%%%%%%%%%%%%%%%%%%%%%%%%%%
%%%%%%%%%%%%%%%%%%%%%%%%%%%%%%%%%%%%%%%%%%%%

Let $(a_m)_{m\in \N}$ and $(b_m)_{m\in \N}$ be complex sequences such that 
$a_0 = b_0 =1$ and denote by $f(x)$ and $g(x)$ their respective exponential generating series, i.e.,
\begin{equation}
f(x)=\sum_{m=0}^{+\infty}\frac{a_m}{m!} \, x^m, 
\qquad 
g(x)=\sum_{m=0}^{+\infty}\frac{b_m}{m!} \, x^m.
\label{DefFG}
\end{equation}
In the following, we will use the notation $[x^n] f(x)$ for the coefficient of 
$x^n$, $n \in \mathbb{N}$, in power series $f(x)$. For all $x \in \mathbb{C}$, define the infinite lower-triangular matrices 
$A(x) = (A_{n,k}(x))_{n,k \in \N^*}$ and 
$B(x) = (B_{n,k}(x))_{n,k \in \N^*}$ by
\begin{equation}
	\left\{
	\begin{array}{ll}
	A_{n,k}(x) = \displaystyle 
	(-1)^k \binom{n}{k}\sum_{m=0}^{n-k}\frac{(k-n)_m \, a_m}{m!} \, x^m,
	\\ \\
	B_{n,k}(x) = \displaystyle 
	(-1)^k \binom{n}{k}\sum_{m=0}^{n-k}\frac{(k-n)_m \, b_m}{m!} \, x^m,
	\end{array} \right.
\label{DefAB}
\end{equation}
where $(c)_m$, $c \in \mathbb{C}$, $m \in \mathbb{N}^*$, denotes the Pochhammer symbol (\cite{NIST10}, §5.2(iii)) with $(c)_0 = 1$. From definition (\ref{DefAB}), matrices $A(x)$ and $B(x)$ have diagonal elements equal to $A_{k,k}(x) = B_{k,k}(x) = (-1)^k$, 
$k \in \mathbb{N}^*$, and are thus invertible. 

%%%%%%%%%%%%%%%%%%%%%%%%%%%%%%%%%%%%%%
\subsection{An inversion criterion}
%%%%%%%%%%%%%%%%%%%%%%%%%%%%%%%%%%%%%%
We first state the following inversion criterion.

\begin{prop}
\textbf{Matrices $A(x)$ and $B(x)$ are inverse of each other if and only if the condition}
	\begin{equation}
	[x^{n-k}]f(-x)g(x)=\mathbf{1}_{n-k}, 
	\qquad 1 \leqslant k \leqslant n,
	\label{Invers0}
	\end{equation}
\textbf{on functions $f$ and $g$ holds (with $\mathbf{1}_{n-k} = 1$ if $n = k$ and 0 otherwise).}
\label{theoMaininversionR}
\end{prop}

\noindent
The proof of Proposition \ref{theoMaininversionR} requires the following technical lemma whose proof is deferred to Appendix \ref{A1}.

\begin{lemma}
	\textbf{Let $N \in \N^*$ and complex numbers $\lambda$, $\mu$. Defining}
	$$
	D_{N}(\lambda,\mu) = 
	\sum_{r=0}^{N-1}\frac{(-1)^r}{\Gamma(1+r-\lambda)\Gamma(1-r+\mu)},
	$$
	\textbf{we then have}
\begin{equation}
	D_{N}(\lambda,\mu) = 
	\left \{
	\begin{array}{ll}
	\displaystyle 
	\frac{1}{\mu - \lambda} \left[ \frac{1}{\Gamma(-\lambda)\Gamma(1+\mu)} - 
	\frac{(-1)^N}{\Gamma(N-\lambda)\Gamma(1-N+\mu)} \right], \, \mu \neq \lambda
	\\ \\
	\displaystyle \frac{\sin(\pi\lambda)}{\pi} \left[ \psi(-\lambda)-\psi(N-\lambda) \right], 
	~ \qquad \qquad \qquad \qquad \quad \, \mu = \lambda,
	\end{array} \right.
\label{Sn}
\end{equation}
\textbf{where $\psi = \Gamma'/\Gamma$.}
	\label{lemm1}
\end{lemma}

\noindent
We now proceed with the justification of Proposition \ref{theoMaininversionR}.

\begin{proof}
$A(x)$ and $B(x)$ being lower-triangular, so is their product 
$C(x) = A(x)B(x)$. After definition (\ref{DefAB}), 
the coefficient 
$C_{n,k}(x) = \sum_{\ell \geqslant 1} A_{n,\ell}(x)B_{\ell,k}(x)$, 
$1 \leqslant k \leqslant n$ (where the latter sum over index $\ell$ is actually finite), of matrix $C(x)$ reads 
\begin{align}
C_{n,k}(x) = & \, 
\sum_{\ell = 1}^{+\infty} (-1)^\ell \, \frac{ n!}{\ell!(n-\ell)!} 
\sum_{m=0}^{n-\ell} \frac{(-1)^m(n-\ell)! \, a_m}{(n-\ell-m)!m!} \, x^m \; \times  
\nonumber \\
& \, (-1)^k \, \frac{\ell!}{k!(\ell-k)!}\sum_{m'=0}^{\ell-k} 
\frac{(-1)^{m'}(\ell-k)! \, b_{m'}}{(\ell-k-m')!m'!} \, x^{m'}
\nonumber
\end{align}
after writing $(-r)_m = (-1)^m r!/(r-m)!$ for any positive integer $r$, that is,
\begin{equation}
C_{n,k}(x) = (-1)^k \frac{n!}{k!} \sum_{\ell = 1}^{+\infty} (-1)^\ell 
\sum_{m = 0}^{n-\ell} \frac{(-1)^m a_m \, x^m}{m!(n-\ell-m)!} 
\sum_{m' = 0}^{\ell - k} \frac{(-1)^{m'} b_{m'} \, x^{m'}}{m'!(\ell-k-m')!}.
\label{P11}
\end{equation}
Exchanging the summation order in (\ref{P11}) further gives
\begin{align}
C_{n,k}(x) = (-1)^k \frac{n!}{k!} & \, 
\sum_{(m,m') \in \Delta_{n,k}} \frac{(-1)^m a_m \, x^m}{m!}
\frac{(-1)^{m'} b_{m'} \, x^{m'}}{m'!} \; \times
\nonumber \\
& \sum_{k \leqslant \ell \leqslant n}\frac{(-1)^\ell}{(n-\ell-m)!(\ell-k-m')!}
\label{P12}
\end{align}
with the subset 
$\Delta_{n,k} = \{(m,m') \in \mathbb{N}^2, \; m+m' \leqslant n-k\}$ for given 
$k \leqslant n$, and where the latter sum on index $\ell$ can be equivalently written as
\begin{align}
\sum_{k \leqslant \ell \leqslant n}\frac{(-1)^\ell}{(n-\ell-m)!(\ell-k-m')!} 
= & \, \sum_{r = 0}^{n-k} \frac{(-1)^{n-r}}{(r-m)!(n-r-k-m')!} 
\nonumber \\
= & \, (-1)^n \, D_{n-k+1}(m,n-k-m')
\nonumber
\end{align}
with the index change $\ell = n-r$ and the notation of Lemma \ref{lemm1}. The expression (\ref{P12}) for coefficient $C_{n,k}(x)$ consequently reduces to
\begin{align}
C_{n,k}(x) = (-1)^{n+k} \, \frac{n!}{k!} \, 
\sum_{(m,m') \in \Delta_{n,k}} & \frac{(-1)^m a_m \, x^m}{m!}
\frac{(-1)^{m'} b_{m'} \, x^{m'}}{m'!} \; \times
\nonumber \\
& D_{n-k+1}(m,n-k-m')
\label{P13}
\end{align}
and we are left to calculate $D_{n-k+1}(m,n-k-m')$ for all non negative 
$m$ and $m'$, $(m,m') \in \Delta_{n,k}$. By Lemma \ref{lemm1} applied to 
$\lambda = m$ and $\mu = n-k-m'$, we successively derive that:

\begin{itemize}
\item[\textbf{(a)}] if $\mu > \lambda \Leftrightarrow m + m' < n-k$, formula 
(\ref{Sn}) entails
\begin{align}
& D_{n-k+1}(m,n-k-m') \; = 
\nonumber \\ 
& \frac{1}{n-k-(m+m')} \left [ \frac{1}{\Gamma(-m)\Gamma(1+n-k-m')} - 
\frac{(-1)^{n-k+1}}{\Gamma(n-k+1-m)\Gamma(-m')} \right ];
\nonumber
\end{align}
as $\Gamma(-m) = \Gamma(-m') = \infty$ for all non negative integers 
$m \geqslant 0$ and $m' \geqslant 0$, each fraction of the latter expression vanishes and thus
\begin{equation}
D_{n-k+1}(m,n-k-m') = 0, \qquad m + m' < n-k;
\label{P14}
\end{equation}

\item[\textbf{(b)}] if $\lambda = \mu \Leftrightarrow m + m' = n-k$, formula 
(\ref{Sn}) entails
\begin{equation}
D_{n-k+1}(m,m) = \lim_{\lambda \rightarrow m} \frac{\sin(\pi\lambda)}{\pi} 
\left [ \psi(-\lambda) - \psi(n-k+1-\lambda) \right ].
\label{P15}
\end{equation}
We have $\sin(m \pi) = 0$ while function $\psi$ has a polar singularity at every non positive integer; the limit (\ref{P15}) is therefore indeterminate 
($0 \times \infty$) but this is solved by invoking the reflection formula 
$$
\psi(z) - \psi(1-z) = - \pi \, \cot(\pi \, z), \qquad z \notin - \mathbb{N},
$$ 
for function $\psi$ (\cite{NIST10}, Chap.5, §5.5.4). In fact, applying the latter to $z = -\lambda$ first gives
$
\sin(\pi \lambda) \, \psi(-\lambda) = 
\sin(\pi \lambda) \, \psi(1+\lambda) + \pi \cdot \cos(\pi\lambda)
$
whence
$$
\lim_{\lambda \rightarrow m} \frac{\sin(\pi \lambda)}{\pi} \, \psi(-\lambda) 
= 0 \times \psi(1+m) + (-1)^m = (-1)^m;
$$
besides, the second term $\psi(n-k+1-\lambda)$ in (\ref{P15}) has a finite limit when $\lambda \rightarrow m$ since $m + m' = n - k \Rightarrow m \leqslant n - k$ so that $n-k+1-\lambda$ tends to a positive integer. From 
(\ref{P15}) and the latter discussion, we are left with
\begin{equation}
D_{n-k+1}(m,m) = (-1)^m, \qquad m + m' = n-k.
\label{P16}
\end{equation}
\end{itemize}

In view of items \textbf{(a)} and \textbf{(b)}, identities 
(\ref{P15}) and (\ref{P16}) together reduce expression (\ref{P13}) to
\begin{align}
C_{n,k}(x) = & \; (-1)^{n+k} \, \frac{n!}{k!} \, 
\sum_{m = 0}^{n-k} \frac{(-1)^m a_m \, x^m}{m!} \, (-1)^{n-k-m} 
\frac{b_{n-k-m} \, x^{n-k-m}}{(n-k-m)!} \times (-1)^m 
\nonumber \\
= & \; \frac{n!}{k!} \, \sum_{m = 0}^{n-k} \frac{(-1)^m a_m \, x^m}{m!} \, 
\frac{b_{n-k-m}}{(n-k-m)!}x^{n-k} = \frac{n!}{k!} \; [x]^{n-k}f(-x)g(x)
\nonumber
\end{align}
where $f$ and $g$ denote the exponential generating function of the sequence 
$(a_m)_{m \in \mathbb{N}^*}$ and the sequence $(b_m)_{m \in \mathbb{N}^*}$, respectively. It follows that $C(x) = A(x) B(x)$ is the identity matrix 
$\mathrm{Id}$ if and only if condition (\ref{Invers0}) holds, as claimed. 
\end{proof}

Following the proof of Proposition \ref{theoMaininversionR}, the same arguments extend to the case when the sequences 
$(a_m)_{m \in \mathbb{N}}$ and $(b_m)_{m \in \mathbb{N}}$ associated with matrices $A(x)$ and $B(x)$ also depend on indexes $n, k$, that is, for  sequences $(a_{m;n,k})_{m \in \mathbb{N}}$ and 
$(b_{m;n,k})_{m \in \mathbb{N}}$. Criterion (\ref{Invers0}) for the inversion relation $A(x) B(x) = \mathrm{Id}$ then simply extends to 
$$
[x^{n-k}]f_{n,k}(-x)g_{n,k}(x) = \mathbf{1}_{n-k},
\qquad 1 \leqslant k \leqslant n,
$$
where $f_{n,k}$ (resp. $g_{n,k}$) denotes the exponential generating function of the sequence $(a_{m;n,k})_{m \in \mathbb{N}}$ (resp. 
$(b_{m;n,k)})_{m \in \mathbb{N}}$) for given $n, \, k \in \mathbb{N}^*$. This straightforward generalization of Proposition \ref{theoMaininversionR} will be hereafter invoked to verify the inversion criterion.

%%%%%%%%%%%%%%%%%%%%%%%%%%%%%%%%%%%%%%
\subsection{The inversion formula}
%%%%%%%%%%%%%%%%%%%%%%%%%%%%%%%%%%%%%%
We now formulate the inversion formula for lower-triangular matrices involving Hypergeometric polynomials.

\begin{theorem}
\textbf{Let $x, \nu \in \mathbb{C}$ and define the lower-triangular matrices 
$A(x,\nu)$ and $B(x,\nu)$ by}
	\begin{equation}
	\left\{
	\begin{array}{ll}
	A_{n,k}(x,\nu) = \displaystyle (-1)^k\binom{n}{k}F(k-n,-n\nu;-n;x),
  \\ \\
	B_{n,k}(x,\nu) = \displaystyle (-1)^k\binom{n}{k}F(k-n,k\nu;k;x)
	\end{array} \right.
	\label{DefABxNU}
	\end{equation}
\textbf{for $1 \leqslant k \leqslant n$. The inversion formula}
	\begin{equation}
	T_n = \sum_{k=1}^{n}A_{n,k}(x,\nu)S_k 
	\Longleftrightarrow 
	S_n=\sum_{k=1}^{n}B_{n,k}(x,\nu)T_k, \quad n \in \mathbb{N}^*,
	\label{eq:inversionR}
	\end{equation}
\textbf{holds for any pair of complex sequences 
$(S_n)_{n \in \mathbb{N}^*}$ and $(T_n)_{n \in \mathbb{N}^*}$.}
	\label{PropIn}
\end{theorem}

\noindent
The factor $F(k-n,-n\nu;-n;x)$ in the definition 
(\ref{DefABxNU}) of matrix $A(x,\nu)$ is always well-defined although the third argument $-n$ is a negative integer; in fact, given 
$1 \leqslant k \leqslant n$, the definition (\cite{NIST10}, 15.2.1) 
\begin{equation}
F(k-n,-n\nu;-n;x) = 
\sum_{m=0}^{n-k} \frac{(k-n)_m(-n\nu)_m}{(-n)_m \, m!} x^m
\label{DefFpoly}
\end{equation}
shows that the denominator $(-n)_m = (-1)^m n!/(n-m)!$ never vanishes for 
all indexes $m \leqslant n-k < n$.

\begin{remark}
The polynomial factors $F(k-n,-n\nu;-n;x)$ and $F(k-n,k\nu;k;x)$ respectively involved in coefficients $A_{n,k}(x,\nu)$ and $B_{n,k}(x,\nu)$ in definition (\ref{DefABxNU}) are deduced from each other by the substitution 
$k \leftrightarrow -n$. This simple substitution, however, does not leave the remaining factor $\binom{n}{k}$ invariant and thus cannot carry out by itself the inversion scheme (\ref{eq:inversionR}). 
\end{remark}

\begin{proof}
To show that $A(x,\nu)B(x,\nu) = \mathrm{Id}$, the Identity matrix, it is sufficient to verify  criterion (\ref{Invers0}). From definition (\ref{DefAB}), we first specify the sequences $(a_{m;n,k})_{m \in \mathbb{N}}$ and 
$(b_{m;n,k})_{m \in \mathbb{N}}$ associated with a given pair $(n,k)$. On one hand, (\ref{DefFpoly}) gives $a_{m;n} = (-n\nu)_m/(-n)_m$, $m \geqslant 0$,
for given $n \in \mathbb{N}^*$ and, in particular, $a_{0;n} = 1$; on the other hand, write
\begin{equation}
F(k-n,k\nu,k;x) = \sum_{m=0}^{n-k} \frac{(k-n)_m(k\nu)_m}{(k)_m \, m!} \, x^m 
\label{DefFFbis}
\end{equation}
so that $b_{m;k} = (k\nu)_m/(k)_m$, $m \geqslant 0$, for given 
$k \in \mathbb{N}^*$ with $b_{0;k} = 1$. Let $f_n$ and $g_{k}$ respectively denote the exponential generating function of these sequences 
$(a_{m;n})_{m \geqslant 0}$ and $(b_{m;k})_{m \geqslant 0}$; the product 
$f_n(-x)g_{k}(x)$ is then given by
\begin{align}
f_n(-x)g_{k}(x) & \, = \left ( \sum_{m \geqslant 0} (-1)^m 
\frac{a_{m;n}}{m!} \, x^m \right ) 
\left ( \sum_{m \geqslant 0} \frac{b_{m;k}}{m!} \, x^m \right )
\nonumber \\
& \, = 
\sum_{m=0}^{+\infty} (-1)^m \frac{(-n\nu)_m}{(-n)_m \, m!} \, x^m \cdot 
\sum_{m=0}^{+\infty} \frac{(k\nu)_m}{(k)_m \, m!} \, x^m = 
\sum_{\ell \geqslant 0} U_\ell^{(n,k)} \, x^\ell
\nonumber
\end{align}
where
\begin{equation}
U_\ell^{(n,k)} = \sum_{m=0}^\ell (-1)^m \frac{(-n\nu)_m}{(-n)_m \, m!} 
\frac{(k\nu)_{\ell-m}}{(k)_{\ell-m} \, (\ell-m)!}, \qquad \ell \geqslant 0.
\label{DefU}
\end{equation}
Let then $n \geqslant k$; from expression (\ref{DefU}), we derive
\begin{align}
U_{n-k}^{(n,k)} = & \, \sum_{m=0}^{n-k} (-1)^m \frac{(-n\nu)_m}{(-n)_m \, m!} \cdot 
\frac{(k\nu)_{n-k-m}}{(k)_{n-k-m} \, (n-k-m)!}
\nonumber \\
= & \, \sum_{m=0}^{n-k} (-1)^m \frac{\Gamma(m-n\nu)}{\Gamma(-n\nu)} 
\cdot \frac{(-1)^m(n-m)!}{n!} \cdot \frac{1}{m!} \cdot 
\frac{\Gamma(n-k-m+k\nu)}{\Gamma(k\nu)} \; \times 
\nonumber \\
& \qquad \; \; \; 
\frac{\Gamma(k)}{\Gamma(n-k-m+k)} \cdot \frac{1}{(n-k-m)!}
\nonumber
\end{align}
after writing the Pochhammer symbol $(c)_m = \Gamma(m+c)/\Gamma(c)$ for 
$c \notin -\mathbb{N}$ and noting that $(-n)_m = (-1)^m n!/(n-m)!$. Reducing the latter expression of $U_{n-k}^{(n,k)}$ gives
\begin{align}
U_{n-k}^{(n,k)} = & \, \frac{\Gamma(k)}{n!\Gamma(-n\nu)\Gamma(k\nu)} 
\sum_{m=0}^{n-k} (n-m) \frac{\Gamma(m-n\nu)\Gamma(n-k-m+k\nu)}{m!(n-k-m)!}
\nonumber \\
= & \, \frac{\Gamma(k)}{n!\Gamma(-n\nu)\Gamma(k\nu)} 
(X_{n-k}^{(n,k)} + Y_{n-k}^{(n,k)})
\label{DefUBIS}
\end{align}
where we introduce the sums (after splitting the difference $n-m$ in the summation (\ref{DefUBIS}) as $k + (n-m-k)$)
$$
\left\{
\begin{array}{ll}
X_{n-k}^{(n,k)} = \displaystyle k \cdot \sum_{m=0}^{n-k} 
\frac{\Gamma(m-n\nu)\Gamma(n-k-m+k\nu)}{m!(n-k-m)!},
\\ \\
Y_{n-k}^{(n,k)} = \displaystyle \sum_{m=0}^{n-k} 
(n-m-k) \cdot \frac{\Gamma(m-n\nu)\Gamma(n-k-m+k\nu)}{m!(n-k-m)!}.
\end{array} \right.
$$
To calculate first $X_{n-k}^{(n,k)}/k$, note that this equals to the coefficient of $x^{n-k}$ in the power series expansion of the product 
$$
\left ( \sum_{m=0}^{+\infty} \frac{\Gamma(m-n\nu)}{m!} \, x^m \right ) 
\left ( \sum_{m=0}^{+\infty} \frac{\Gamma(m+k\nu)}{m!} \, x^m \right ) = 
\frac{\Gamma(-n\nu)}{(1-x)^{-n\nu}} \cdot \frac{\Gamma(k\nu)}{(1-x)^{k\nu}}
$$
so that
\begin{equation}
X_{n-k}^{(n,k)} = k \, \Gamma(-n\nu)\Gamma(k\nu) \cdot 
[x]^{n-k} \left \{ \frac{(1-x)^{n\nu}}{(1-x)^{k\nu}} \right \}.
\label{DefUTER}
\end{equation}
As to the second sum $Y_{n-k}^{(n,k)}$, it equals the coefficient of $x^{n-k}$ in the power series expansion of the product 
$$
\left ( \sum_{m=0}^{+\infty} \frac{\Gamma(m-n\nu)}{m!} \, x^m \right ) \cdot 
x \, \frac{\mathrm{d}}{\mathrm{d}x} 
\left [ \frac{\Gamma(k\nu)}{(1-x)^{k\nu}} \right ] = 
\frac{\Gamma(-n\nu)}{(1-x)^{-n\nu}} \times x \, 
\Gamma(k\nu) \, \frac{k\nu}{(1-x)^{k\nu+1}}
$$
so that
\begin{equation}
Y_{n-k}^{(n,k)} = \Gamma(-n\nu)\Gamma(k\nu+1) \cdot [x]^{n-k} 
\left \{ \frac{x(1-x)^{n\nu}}{(1-x)^{k\nu+1}} \right \}.
\label{DefUQUATER}
\end{equation}
Using formulas (\ref{DefUTER}) and (\ref{DefUQUATER}) for sums 
$X_{n-k}^{(n,k)}$ and $Y_{n-k}^{(n,k)}$, the expression (\ref{DefUBIS}) for 
$U_{n-k}^{(n,k)}$ then easily reduces to
\begin{align}
U_{n-k}^{(n,k)} = & \, \frac{[x]^{n-k}}{n!} \left \{ \Gamma(k+1) 
\frac{(1-x)^{n\nu}}{(1-x)^{k\nu}} + k \nu \, \Gamma(k) 
\frac{x(1-x)^{n\nu}}{(1-x)^{k\nu+1}} \right \}
\nonumber \\
= & \, \frac{k!}{n!} \, 
\Bigl \{ [x^{n-k}] (1-x)^{(n-k)\nu-1} (1+(\nu-1)x) \Bigr \}, 
\qquad n \geqslant k.
\label{DefU5}
\end{align}
With the series expansion $(1-x)^{(n-k)\nu-1} = \sum_{\ell \geqslant 0} x^\ell 
(1-(n-k)\nu)_\ell/\ell!$, expression (\ref{DefU5}) for $n-k \geqslant 1$ then gives
\begin{align}
U_{n-k}^{(n,k)} & \, = \frac{k!}{n!} 
\left \{ \frac{(1-(n-k)\nu)_{n-k}}{(n-k)!} + 
(\nu - 1) \frac{(1-(n-k)\nu)_{n-k-1}}{(n-k-1)!} \right \} 
\nonumber \\
& \, = \frac{k!}{n!(n-k)!} \, \frac{V_{n,k}}{\Gamma(1-(n-k)\nu)}
\nonumber
\end{align}
where factor $V_{n,k}$ is given by
$$
V_{n,k} = \Gamma(1-(n-k)\nu +n-k) + (\nu-1)\Gamma((n-k)(1-\nu))(n-k);
$$
the relation $\Gamma(1+z) = z \Gamma(z)$ applied to 
$z = (n-k)\nu +n-k = (n-k)(1-\nu)$ readily entails that $V_{n,k} = 0$ hence 
$$
U_{n-k}^{(n,k)} = 0, \qquad n - k \geqslant 1.
$$
Now if $n = k$, expression (\ref{DefU5}) reduces to
$$
U_{n-k}^{(n,k)} = [x^0] \left \{ 1 + \frac{\nu \, x}{1-x} \right \} = 1.
$$
The inversion condition (\ref{Invers0}) for 
$U_{n-k}^{(n,k)} = [x]^{n-k}f_n(-x)g_{k}(x) = \mathbf{1}_{n-k}$ is therefore fulfilled for all $n, \, k \geqslant 1$ and we conclude that inverse relation 
(\ref{eq:inversionR}) holds for any pair of sequences $(S_n)_{n \geqslant 1}$ and $(T_n)_{n \geqslant 1}$.
\end{proof}

%%%%%%%%%%%%%%%%%%%%%%%%%%%%%%%%%%%%%%%%%%%%%%%
%%%%%%%%%%%%%%%%%%%%%%%%%%%%%%%%%%%%%%%%%%%%%%%

\section{Generating functions}
\label{GF}

%%%%%%%%%%%%%%%%%%%%%%%%%%%%%%%%%%%%%%%%%%%%%%%
%%%%%%%%%%%%%%%%%%%%%%%%%%%%%%%%%%%%%%%%%%%%%%%

As a direct consequence of Theorem \ref{PropIn}, remarkable functional relations can be derived for the ordinary (resp. exponential) generating functions of sequences related by the inversion formula. 

%%%%%%%%%%%%%%%%%%%%%%%%%%%%%%%%%%%%%%%%%%%%%
\subsection{Relations for ordinary G.F.'s}
%%%%%%%%%%%%%%%%%%%%%%%%%%%%%%%%%%%%%%%%%%%%%
We first address ordinary generating functions and state the following reciprocal relations.

\begin{corol}
\textbf{For given complex parameters $x$ and $\nu$, let 
$(S_n)_{n \in \mathbb{N}^*}$ and $(T_n)_{n \in \mathbb{N}^*}$ be sequences related by the inversion formulas (\ref{eq:inversionR}) of Theorem 
\ref{PropIn}, that is, 
$S = B(x,\nu) \cdot T \Leftrightarrow T = A(x,\nu) \cdot S$.}

\textbf{Denote by $\mathfrak{G}_S(z)$ and $\mathfrak{G}_T(z)$ the formal ordinary generating series of $S$ and $T$, respectively. Defining the mapping 
$\Xi$ (depending on parameters $x$ and $\nu$) by}
\begin{equation}
\Xi(z) = \frac{z}{z-1}\Bigl(\frac{1-z}{1-z(1-x)}\Bigr)^{\nu},
\label{DefXi}
\end{equation}
\textbf{the relations}
\begin{equation}
\mathfrak{G}_S(z) = \left [ \frac{1-\nu}{1-z} + \frac{\nu}{1-z(1-x)} \right ]
\mathfrak{G}_T(\Xi(z))
\label{eq:OGFRelation}
\end{equation}
\textbf{and}
\begin{equation}
\mathfrak{G}_T(\xi) = \mathfrak{G}_S(\Omega(\xi)) 
\left [ \frac{1-\nu}{1-\Omega(\xi)} + 
\frac{\nu}{1-(1-x)\Omega(\xi)} \right ]^{-1}
\label{eq:OGFRelationBIS}
\end{equation}
\textbf{hold, where $\Omega$ is the inverse mapping 
$\Xi(z) = \xi \Leftrightarrow z = \Omega(\xi)$.} 
\label{corOGF}
\end{corol}

\begin{proof}
\textbf{a)} From the definition (\ref{DefABxNU}) of matrix $B(x,\nu)$, the generating function of the sequence $S = B(x,\nu) \cdot T$ is given by
\begin{align}
\mathfrak{G}_S(z) & \, = \sum_{n \geqslant 1} z^n 
\left ( \sum_{k=1}^n B_{n,k}(x,\nu) T_k \right ) = 
\left ( \sum_{k=1}^n (-1)^k \frac{n!}{k!(n-k)!} F(k-n,k\nu;k;x) T_k \right ) 
\nonumber \\
& \, = \sum_{k \geqslant 1} (-1)^k T_k \frac{z^k}{k!} \sum_{n \geqslant k} 
\frac{n!}{(n-k)!} F(k-n,k\nu;k;x) \, z^{n-k}
\nonumber
\end{align}
after changing the summation order; using the expression (\ref{DefFFbis}) for the Hypergeometric coefficient $F(k-n,k\nu;k;x)$, we then obtain
\begin{align}
\mathfrak{G}_S(z) & = \sum_{k \geqslant 1} (-1)^k T_k \frac{z^k}{k!} 
\sum_{n \geqslant k} \frac{n! \, z^{n-k}}{(n-k)!} \; \times 
\nonumber \\
& \quad \sum_{m=0}^{n-k} \frac{(-1)^m(n-k)!}{(n-k-m)!}
\frac{\Gamma(m+k\nu)}{\Gamma(k\nu)} \frac{(k-1)!}{(m+k-1)!}\frac{x^m}{m!}
\nonumber \\
& = \sum_{k \geqslant 1} (-1)^k T_k \frac{z^k}{k} \sum_{n \geqslant k} 
n! \, z^{n-k} \sum_{m=0}^{n-k} \frac{(-1)^m}{(n-k-m)!}
\frac{\Gamma(m+k\nu) \, x^m}{\Gamma(k\nu)m!} \frac{1}{(m+k-1)!}
\nonumber
\end{align}
and the index change $n = k + r$, $r \geqslant 0$, yields
\begin{align}
\mathfrak{G}_S(z) & = \sum_{k \geqslant 1} (-1)^k T_k \frac{z^k}{k} 
\sum_{r \geqslant 0} (k+r)! \, z^{r} \sum_{m=0}^{r} \frac{(-1)^m}{(n-k-m)!}
\frac{(k\nu)_m \, x^m}{m!} \frac{1}{(m+k-1)!}
\nonumber \\
& = \sum_{k \geqslant 1} (-1)^k T_k \frac{z^k}{k} 
\sum_{m \geqslant 0} (-1)^m \, \frac{(k\nu)_m \, x^m}{m!} \frac{1}{(m+k-1)!}
\left ( \sum_{r = m}^{+\infty} \frac{(k+r)!}{(r-m)!} \, z^{r} \right )
\nonumber
\end{align}
where the last sum on index $r$ readily equals
$$
\sum_{r = m}^{+\infty} \frac{(k+r)!}{(r-m)!} \, z^{r} = 
\sum_{r = 0}^{+\infty} \frac{(k+m+r)!}{r!} \, z^{r+m} = 
\frac{(m+k)!}{(1-z)^{k+m+1}} \cdot z^m, \qquad \vert z \vert < 1.
$$
The latest expression of $\mathfrak{G}_S(z)$ consequently reads
\begin{align}
\mathfrak{G}_S(z) = & \, \sum_{k \geqslant 1} 
(-1)^k T_k \frac{z^k}{k} \frac{1}{(1-z)^{k+1}} 
\sum_{m \geqslant 0} (-1)^m \, \frac{(k\nu)_m}{m!} 
\left ( \frac{x\, z}{1-z} \right)^m (m+k)
\nonumber \\
= & \; \frac{1}{1-z} 
\sum_{k \geqslant 1} \frac{T_k}{k} \left ( \frac{z}{z-1} \right )^k
\Bigl [ - \frac{xz}{1-z} 
\sum_{m \geqslant 0} m \left ( \frac{- xz}{1-z} \right )^{m-1} 
\frac{(k\nu)_m}{m!}
\nonumber \\
& \qquad \qquad \qquad \qquad \qquad \; \; + 
k \times \sum_{m \geqslant 0} \left ( \frac{- xz}{1-z} \right )^{m-1} 
\frac{(k\nu)_m}{m!} \Bigr ].
\label{Sum1}
\end{align}
Using successively identity 
$\sum_{m \geqslant 0} (k\nu)_m Z^m/m! = 1/(1-Z)^{k\nu}$ and its 
term-to-term derivative 
$\sum_{m \geqslant 0} m (k\nu)_m Z^{m-1}/m! = k\nu/(1-Z)^{k\nu + 1}$ with respect to $Z$, the sum (\ref{Sum1}) reduces to
\begin{align}
\mathfrak{G}_S(z) = & \, \frac{1}{1-z} \left [ \frac{- xz}{1-z} 
\left ( \frac{1-z}{1-(1-x)z} \right ) \nu \cdot \mathfrak{G}_T(\Xi(z)) + \mathfrak{G}_T(\Xi(z)) \right ]
\nonumber \\
= & \, \frac{1}{1-z} \left [ \frac{- \nu xz}{1-(1-x)z} + 1 \right ] 
\mathfrak{G}_T(\Xi(z))
\nonumber
\end{align}
with $\Xi(z)$ defined as in (\ref{DefXi}). Writing 
$$
\frac{1}{1-z} \left [ \frac{- \nu xz}{1-(1-x)z} + 1 \right ] = 
\frac{1-\nu}{1-z} + \frac{\nu}{1-z(1-x)}
$$
eventually entails relation (\ref{eq:OGFRelation}). 

\textbf{b)} For any parameters $x$ and $\nu$, the function 
$z \mapsto \Xi(z)$ is analytic in a neigborhood of $z = 0$, with 
$\Xi(0) = 0$ and $\Xi'(z) \sim -z$ as $z \downarrow 0$, hence 
$\Xi'(0) = -1 \neq 0$. By the Implicit Function Theorem, $\Xi$ has an analytic inverse $\Omega:\xi \mapsto \Omega(\xi)$ in a neighborhood of 
$\xi = 0$ and the inversion of (\ref{eq:OGFRelation}) provides 
(\ref{eq:OGFRelationBIS}), as claimed. 
\end{proof}

\noindent 
Relation (\ref{eq:OGFRelationBIS}) between formal generating series can also be understood as a functional identity between the analytic functions  
$z \mapsto \mathfrak{G}_S(z)$ and $z \mapsto \mathfrak{G}_T(z)$ in some neighborhood of the origin $z = 0$ in the complex plane. Now, Corollary 
\ref{corOGF} can be supplemented by making explicit the inverse mapping 
$\Omega$ involved in the reciprocal relation (\ref{eq:OGFRelationBIS}). To this end, we state some preliminary properties (in the sequel, $\log$ will denote the determination of the logarithm in the complex plane cut along the negative semi-axis $]-\infty,0]$ with $\log(1) = 0$).

\begin{lemma}
\textbf{Let $R(\nu) = \vert e^{-\psi(\nu)} \vert$ where}
$$
\psi(\nu) = \left\{
\begin{array}{ll}
(1-\nu)\log(1-\nu) + \nu\log(-\nu), \quad \; 
\nu \in \mathbb{C} \setminus \, [0,+\infty[,
\\ \\
(1-\nu)\log(1-\nu) + \nu\log(\nu), \; \quad \; \; \, 
\nu \in \mathbb{R}, \; 0 \leqslant \nu < 1, 
\\ \\
(1-\nu)\log(\nu-1) + \nu\log(\nu), \; \quad \; \; 
\nu \in \mathbb{R}, \; \nu \geqslant 1.
\end{array} \right.
$$
\textbf{The power series}
$$
\pmb{\Sigma}(w) = 
\sum_{b \geqslant 1} 
\frac{\Gamma(b(1-\nu))}{\Gamma(b)\Gamma(1-b\nu)} \cdot w^b, 
\qquad \vert w \vert < R(\nu),
$$
\textbf{is given by}
\begin{equation}
\pmb{\Sigma}(w) = \frac{\Theta(w)-1}{\nu \, \Theta(w) + 1 - \nu}
\label{U0}
\end{equation}
\textbf{where $\Theta:w \mapsto \Theta(w)$ denotes the unique analytic solution 
(depending on $\nu$) to the implicit equation}
\begin{equation}
1 - \Theta + w \cdot \Theta^{1-\nu} = 0, 
\qquad \vert w \vert < R(\nu),
\label{DefTheta}
\end{equation}
\textbf{verifying $\Theta(0) = 1$.}

%%%%%%%%%%%%%%%%%%%%%%%%%%%%%%%%%%%%%%%
\begin{comment}
%%%%%%%%%%%%%%%%%%%%%%%%%%%%%%%%%%%%%%%
\textbf{b) Function $\pmb{\Sigma}$ is the solution to the first order differential equation}
\begin{align}
w \, \pmb{\Sigma}'(w) = & \, \pmb{\Sigma}(w) 
\left [1 - \nu \, \pmb{\Sigma}(w) \right ]  
\left [1 + (1-\nu)\pmb{\Sigma}(w) \right ]
\nonumber \\
= & \, \pmb{\Sigma}(w) \left [ 1 + (1-2\nu)\pmb{\Sigma}(w) - 
\nu(1-\nu)\pmb{\Sigma}(w)^2 \right ]
\label{EDiff0}
\end{align}
\textbf{with initial condition $\pmb{\Sigma}(0) = 0$.}
%%%%%%%%%%%%%%%%%%%%%%%%%%%%%%%%%%%%%%%%
\end{comment}
%%%%%%%%%%%%%%%%%%%%%%%%%%%%%%%%%%%%%%%%
\label{lemmU}
\end{lemma}

\noindent 
The proof of Lemma \ref{lemmU} is detailed in Appendix \ref{A3}. 
% Quite remarkably, function $\pmb{\Sigma}$ will also prove useful in Section \ref{SecA} for the derivation of the generating function of the solution $E$ to the particular system (\ref{T0}). 

\begin{corol}
\textbf{For all $\nu \in \mathbb{C}$ and $x \neq 0$, the inverse mapping 
$\Omega$ of $\Xi$ defined in (\ref{DefXi}) can be expressed by}
\begin{equation}
\Omega(\xi) = 
\frac{\pmb{\Sigma}(x \, \xi)}{(1-x(1-\nu)) \, \pmb{\Sigma}(x \, \xi) - x}, 
\qquad \vert \xi \vert < \frac{R(\nu)}{\vert x \vert},
\label{InvXI}
\end{equation}
\textbf{in terms of power series $\pmb{\Sigma}(\cdot)$ defined in Lemma 
\ref{lemmU}.}
\label{corOGFbis}
\end{corol}

\begin{proof}
\textbf{\textit{(i)}} The homographic transform $h:z \mapsto \theta$ with
$\theta = (1-z)/(1-z(1-x))$ is an involution, with inverse $h^{-1}$ given by
\begin{equation}
z = h^{-1}(\theta) = \frac{1-\theta}{1-\theta(1-x)}.
\label{MobInv}
\end{equation}
Let then $\xi = \Xi(z)$ with function $\Xi$ defined as in (\ref{DefXi}); we first claim that the corresponding $\theta = h(z)$ equals 
$\theta = \Theta(x \, \xi)$ where $\Theta$ is the function defined by the implicit equation (\ref{DefTheta}). In fact, definition (\ref{DefXi}) for 
$\Xi$ and expression (\ref{MobInv}) for $z$ in terms of $\theta$ together entail
$$
\xi = \Xi(z) = \frac{z}{z-1} \, \theta^{\, \nu} = 
\displaystyle \frac{1-\theta}{1-\theta(1-x)} 
\left ( \frac{1-\theta}{1-\theta(1-x)} - 1 \right )^{-1} \, \theta^{\, \nu} = 
\frac{\theta - 1}{x \, \theta} \, \theta^{\, \nu} 
$$
and the two sides of the latter equalities give 
$1 - \theta + x \xi \theta^{1-\nu} = 0$, hence the identity 
$\theta = \Theta(x \, \xi)$, as claimed. 

\textbf{\textit{(ii)}} The corresponding inverse $z = \Omega(\xi)$ can now be expressed as follows; equality (\ref{U0}) applied to $w = x \, \xi$ can be first solved for $\Theta(x \, \xi)$, giving
$$
\Theta(x \, \xi) = \frac{1 + (1-\nu)\pmb{\Sigma}(x \xi)}
{1 - \nu \, \pmb{\Sigma}(x \xi)};
$$
it then follows from (\ref{MobInv}) and this expression of $\Theta(x \, \xi)$ that
$$
z = \Omega(\xi) = \frac{1 - \Theta(x \, \xi)}{1 - (1-x)\Theta(x \, \xi)} = 
\frac{\displaystyle 1 - \frac{1 + (1-\nu)\pmb{\Sigma}(x \xi)}
{\displaystyle 1 - \nu \, \pmb{\Sigma}(x \xi)}}
{\displaystyle 1 - (1-x)\frac{1 + (1-\nu)\pmb{\Sigma}(x \xi)}
{1 - \nu \, \pmb{\Sigma}(x \xi)}}
$$
which easily reduces to formula (\ref{InvXI}).
\end{proof}

\begin{remark}
It can be shown (see Appendix \ref{A3}.c)) that function $\pmb{\Sigma}$ is the solution to the first order non-linear differential equation
%\begin{align}
\begin{equation}
w \, \pmb{\Sigma}'(w) = \pmb{\Sigma}(w) 
\left [1 - \nu \, \pmb{\Sigma}(w) \right ]  
\left [1 + (1-\nu)\pmb{\Sigma}(w) \right ]
%\nonumber \\
%= & \, \pmb{\Sigma}(w) \left [ 1 + (1-2\nu)\pmb{\Sigma}(w) - \nu(1-\nu)\pmb{\Sigma}(w)^2 \right ]
\label{EDiff0}
%\end{align}
\end{equation}
with initial condition $\pmb{\Sigma}(0) = 0$ (so that $\pmb{\Sigma}'(0) = 1$).
\end{remark}

%%%%%%%%%%%%%%%%%%%%%%%%%%%%%%%%%%%%%%%%%%%%%
\subsection{Relations for exponential G.F.'s}
%%%%%%%%%%%%%%%%%%%%%%%%%%%%%%%%%%%%%%%%%%%%%
We now turn to the derivation of identities between the exponential generating functions of any pair of related sequences $S$ and $T$.

\begin{corol}
\textbf{Given sequences $S$ and $T$ related by the inversion formulae 
$S = B(x,\nu) \cdot T \Leftrightarrow T = A(x,\nu) \cdot S$, the exponential generating function $\mathfrak{G}_S^*$ of the sequence $S$ can be expressed by}
\begin{equation}
\mathfrak{G}_S^*(z) = \exp(z) \cdot 
\sum_{k \geqslant 1} (-1)^k T_k \, \frac{z^k}{k!} \, \Phi(k\nu;k;-x \, z), 
\qquad z \in \mathbb{C},
\label{GsExp}
\end{equation}
\textbf{where $\Phi(\alpha;\beta;\cdot)$ denotes the Confluent Hypergeometric function with parameters $\alpha$, $\beta \notin -\mathbb{N}$.}
\label{corEGF}
\end{corol}

\begin{proof}
A calculation similar to that of Corollary \ref{corOGF} gives
\begin{align}
\mathfrak{G}_S^*(z) = & \, \sum_{n \geqslant 0} \frac{z^n}{n!} 
\left ( \sum_{k=1}^n B_{n,k}(x,\nu) T_k \right )
\nonumber \\
= & \, \sum_{k \geqslant 1} (-1)^k T_k \frac{z^k}{k} 
\sum_{m \geqslant 0} (-1)^m \, \frac{\Gamma(m+k\nu) \, x^m}{\Gamma(k\nu)m!} \frac{1}{(m+k-1)!}
\left ( \sum_{r = m}^{+\infty} \frac{z^r}{(r-m)!} \right );
\nonumber
\end{align}
as $\sum_{r \geqslant m} z^r/(r-m)! = z^m \exp(z)$, the latter reduces to
$$
\mathfrak{G}_S^*(z) = \exp(z) \sum_{k \geqslant 1} (-1)^k T_k \frac{z^k}{k} 
\sum_{m \geqslant 0} (-xz)^m \, \frac{(k\nu)_m}{m!} 
\frac{1}{(k-1)!(k)_m}
$$
which, from the expansion of $\Phi(k\nu;k;-xz)$ in powers of $-xz$, yields  
(\ref{GsExp}).
\end{proof}

%\noindent
% Expression (\ref{GsExp}), however, does not directly relate to the exponential generating function $\mathfrak{G}_T^*$ of the sequence $T$. 
Reversely, we have not been able to obtain a remarkable identity for the exponential generating function $\mathfrak{G}_T^*$ in terms of 
$\mathfrak{G}_S^*$. 

%%%%%%%%%%%%%%%%%%%%%%%%%%%%%%%%%%%%%%%%%%%%%%%%%%
%%%%%%%%%%%%%%%%%%%%%%%%%%%%%%%%%%%%%%%%%%%%%%%%%%

\section{Inversion of operator $\mathfrak{L}$}
\label{SecA}

%%%%%%%%%%%%%%%%%%%%%%%%%%%%%%%%%%%%%%%%%%%%%%%%%%
%%%%%%%%%%%%%%%%%%%%%%%%%%%%%%%%%%%%%%%%%%%%%%%%%%

We first apply (Section \ref{SecA1}) the inversion formula of Theorem 
\ref{PropIn} to the resolution of the infinite linear system (\ref{T0})  formulated in the Introduction. The associated relation between exponential generating functions (Corollary \ref{corEGF}) further provides an integral representation for the solution $E^*$ to the integral equation (\ref{EI0}), hence for the inverse $\mathfrak{L}^{-1}$ of integro-differential operator 
$\mathfrak{L}$ introduced in (\ref{DefTT}).

Operator $\mathfrak{L}$ has been introduced for real parameters 
$x \in \; ]0,1[$ and $\nu < 0$; as per the discussion of previous Sections \ref{LTS} and \ref{GF} where complex parameters have been generally considered, we now extend definition (\ref{DefTT}) of $\mathfrak{L}$ to complex parameters 
\begin{itemize}
\item $x \in \; \mathbb{C} \setminus 
( \mathbb{R}^- \cup \{1\})$ (so that $1/(1-x)$ is finite and does not belong to the integration interval $[0,1]$) 
\item and $\nu \in \mathbb{C}$ such that $\mathrm{Re}(\nu) < 0$. 
\end{itemize}
Within these assumptions, it is easily verified that 
$\mathfrak{L}(\mathscr{H}_0) \subset \mathscr{H}_0$ where 
$\mathscr{H}_0$ is again the linear space of entire functions in 
$\mathbb{C}$ vanishing at $0$.

%%%%%%%%%%%%%%%%%%%%%%%%%%%%%%%%%%%%%%%%%%%%%%%%%%%%%%%%%%%%%%%%%%%%%%%
\subsection{Resolution of system (\ref{T0})}
\label{SecA1}
%%%%%%%%%%%%%%%%%%%%%%%%%%%%%%%%%%%%%%%%%%%%%%%%%%%%%%%%%%%%%%%%%%%%%%%
We have claimed in the Introduction (Section \ref{IM}.\textbf{II}) that 
integro-differential equation (\ref{EI0}) reduces to the infinite system 
(\ref{T0}). We first justify this assertion by showing, in particular, how the coefficients of system (\ref{T0}) can be eventually expressed in terms of Hypergeometric polynomials.

\begin{prop}
\textbf{Reduction Property (II) holds: equation (\ref{EI0}) reduces to system (\ref{T0}) with matrix $Q = (Q_{b,\ell})_{1 \leqslant \ell  \leqslant b}$ related to Hypergeometric polynomials as in (\ref{Q0}).}
\label{LemT0}
\end{prop}

\begin{proof}
To derive system (\ref{T0}), we expand both sides of (\ref{EI0}) into power series of variable $z$ and identify like powers on each side. The series expansion (\ref{defE*}) of $E^*(Z)$ in powers of $Z$ first provides
\begin{equation}
(1 + Z)E^*(\zeta Z) - c \, Z \, \frac{\mathrm{d} E^*}{\mathrm{d} z}(\zeta Z) = 
\sum_{b \geqslant 1} \Lambda_b(\zeta) \frac{{Z}^b}{b!}
\label{Bsv00}
\end{equation}
where we set
$\Lambda_b(\zeta) = \zeta^b E_b + b \, \zeta^{b-1} 
E_{b-1} - b \, c \zeta^{b-1}E_b$ 
for all $\zeta$ and with the constant $c = (1-\nu x)/(1-x)$; applying equality (\ref{Bsv00}) to the argument $Z = \mathfrak{R}(\zeta) \cdot z$, the integrand of $\mathfrak{L} E^*(z)$ in (\ref{DefTT}) can then be expanded into a power series of $z$ as
\begin{equation}
\mathfrak{L} E^*(z) = \int_0^{1} \biggl[ \sum_{b \geqslant 1} 
\Lambda_b(\zeta) \, \frac{\zeta^b \, \mathfrak{R}(\zeta)^b \, z^b}{b!} \biggr]
\, e^{-\mathfrak{R}(\zeta) \, z} \, \mathrm{d}\zeta.
\label{Bsv0}
\end{equation}
Now, expanding the exponential $e^{-\mathfrak{R}(\zeta) \, z}$ of the integrand in (\ref{Bsv0}) into a power series of $z$ gives the expansion
\begin{equation}
\mathfrak{L} E^*(z) = \sum_{b \geqslant 0} (-1)^b \frac{z^b}{b!} \, 
\sum_{\ell=0}^b (-1)^{\ell} \binom{b}{\ell} \int_0^{U} 
\zeta^\ell \Lambda_\ell(\zeta) \mathfrak{R}(\zeta)^b \, \mathrm{d}\zeta
\label{Bsv}
\end{equation}
(after noting that $\Lambda_0(\zeta) = 0$ since $E_0 = 0$ by 
definition). On account of expansion (\ref{Bsv}) with the above definition (\ref{Bsv00}) of $\Lambda_\ell(\zeta)$, together with the expansion (\ref{defE*}) for $K(z)$, the identification of like powers of these expansions readily yields the relation
\begin{equation}
\sum_{\ell=1}^b (-1)^\ell \binom{b}{\ell} B_{b,\ell} E_\ell^* + 
\sum_{\ell=1}^b (-1)^\ell \binom{b}{\ell} \ell \, M_{b,\ell-1} 
E_{\ell-1} = K_b, \qquad b \geqslant 1,
\label{T0bis}
\end{equation}
with $B_{b,\ell} = M_{b,\ell} - \ell \, c \, M_{b,\ell-1}$, 
where $M_{b,\ell}$ denotes the definite integral
\begin{equation}
M_{b,\ell} = \int_0^1 \zeta^\ell \, \mathfrak{R}(\zeta)^b \, \mathrm{d}\zeta, 
\qquad 1 \leqslant \ell \leqslant b.
\label{DefMbell}
\end{equation}
By first changing the index in the second sum in the left-hand side of 
(\ref{T0bis}) and then using identity 
$\binom{b}{\ell+1} = (b-\ell) \cdot \binom{b}{\ell}/(\ell + 1)$, 
(\ref{T0bis}) reduces to (\ref{T0}) with coefficients
\begin{equation}
Q_{b,\ell} = (\ell + 1 - b) M_{b,\ell} - 
\ell \, c \, M_{b,\ell-1}, \qquad 1 \leqslant \ell \leqslant b.
\label{DefQbell}
\end{equation}
The calculation of integral $M_{b,\ell}$ in (\ref{DefMbell}) in terms of Hypergeometric functions and its reduction to Hypergeometric polynomials is detailed in Appendix \ref{A4}; this eventually provides expression 
(\ref{Q0}) for the coefficients of matrix $Q = (Q_{b,\ell})$. 
\end{proof}

We can now deduce the unique solution to system (\ref{T0}).

\begin{corol}
\textbf{Let $\nu \in \mathbb{C}$ with $\mathrm{Re}(\nu) < 0$. Given the sequence $(K_b)_{b \geqslant 1}$, the unique solution 
$(E_b)_{b \geqslant 1}$ to system (\ref{T0}) is given by}
\begin{equation}
E_b = (1-x) \sum_{\ell = 1}^b (-1)^{\ell - 1} 
\binom{b}{\ell} F(\ell-b,\ell \nu;\ell;x) \, x^{\ell - 1} 
\frac{\Gamma(\ell - \ell \nu)}{\Gamma(\ell)\Gamma(1-\ell \nu)} \, K_\ell
\label{E0}
\end{equation}
\textbf{for all $b \geqslant 1$.}
\label{C2}
\end{corol}

\begin{proof}
By expression (\ref{Q0}) for the coefficients of lower-triangular matrix $Q$, equation (\ref{T0}) equivalently reads
\begin{equation}
\sum_{\ell=1}^b (-1)^\ell \binom{b}{\ell} F(\ell-b,-b\nu;-b;x) \cdot 
E_\ell = \widetilde{K}_b, \qquad 1 \leqslant \ell \leqslant b,
\label{S0}
\end{equation}
when setting
\begin{equation}
\widetilde{K}_b = \displaystyle - \, \frac{\Gamma(b - b \nu)}
{\Gamma(b)\Gamma(1-b \nu)} (1-x)x^{b-1} \cdot K_b, 
\qquad b \geqslant 1.
\label{S1}
\end{equation}
The application of inversion Theorem \ref{PropIn} to lower-triangular system (\ref{S0}) readily provides the solution sequence 
$(E_\ell)_{\ell \in \mathbb{N}}$ in terms 
of the sequence $(\widetilde{K}_b)_{b \in \mathbb{N}^*}$; using then  transformation (\ref{S1}), the final solution (\ref{E0}) for the sequence 
$(E_\ell)_{\ell \in \mathbb{N}^*}$ follows. 
\end{proof}

%%%%%%%%%%%%%%%%%%%%%%%%%%%%%%%%%%%%%%%%%%%%%%%%%%%%%%
\subsection{Inversion of operator $\mathfrak{L}$}
\label{SecA2}
%%%%%%%%%%%%%%%%%%%%%%%%%%%%%%%%%%%%%%%%%%%%%%%%%%%%%%
The factorization property asserted in the Introduction for operator 
$\mathfrak{L}$ is proved in the following.

\begin{prop}
\textbf{Factorization Property (I) holds: the linear operator $\mathfrak{L}$ on space $\mathscr{H}_0$ can be factorized as in (\ref{DefTTbis}), in terms of operators $\delta = z \, \mathrm{d}/\mathrm{d}z$ and $\mathfrak{M}$.}
\label{FactorP}
\end{prop}

\begin{proof}
Calculating the exponential generating function of the sequence 
$(-1)^bK_b$, $b \geqslant 1$, from relation (\ref{T0}) with help of 
(\ref{Q0}) for the coefficients of matrix $Q$ gives
\begin{align}
& \mathfrak{L}E^*(z) = K(z) = \sum_{b \geqslant 1} (-1)^b K_b \, z^b \; = 
\nonumber \\
& \sum_{b \geqslant 1} - \frac{\Gamma(b)\Gamma(1-b\nu)}{\Gamma(b-b\nu)} \, \frac{x^{1-b}}{1-x} \, \frac{(-z)^b}{b!} \sum_{\ell = 1}^b (-1)^\ell \binom{b}{\ell} \, F(\ell-b,-b\nu;-b;x) E_\ell,
\nonumber
\end{align}
for all $z \in \mathbb{C}$, that is,
\begin{align}
\mathfrak{L}E^*(z) = 
& \sum_{\ell \geqslant 1} \frac{(-1)^\ell}{\ell!} E_\ell \; \times 
\nonumber \\
& \sum_{b \geqslant \ell} - \frac{\Gamma(b)\Gamma(1-b\nu)}{\Gamma(b-b\nu)} \, 
\frac{x^{1-b}}{1-x} \, \frac{(-z)^b}{(b-\ell)!} F(\ell-b,-b\nu;-b;x)
\label{F0}
\end{align}
(after changing the summation order on indexes $b$ and $\ell$). Applying the general identity (\ref{IDFG}) to parameters $m = b-\ell \geqslant 0$, 
$\beta = -b\nu$ and $\gamma = \ell - b\nu +1$ to express polynomial 
$F(\ell-b,-b\nu,-b;x)$ in terms of polynomial 
$F(\ell-b,-b\nu,\ell-b\nu+1;1-x)$, we further obtain
\begin{align}
& - \frac{\Gamma(b)\Gamma(1-b\nu)}{\Gamma(b - b\nu)} \, F(\ell-b,-b\nu,-b;x) 
\; = 
\nonumber \\
& - \frac{(1-\nu)\Gamma(\ell+1)\Gamma(1-b\nu)}{\Gamma(\ell-b\nu+1)} \, 
F(\ell-b,-b\nu,\ell-b\nu+1;1-x);
\label{F1}
\end{align}
using the integral representation recalled in Appendix \ref{A4} - 
Equ.(\ref{HyperGauss}) for the factor $F(\ell-b,-b\nu,\ell-b\nu+1;1-x)$ in the right-hand side of (\ref{F1}) eventually yields 
$$
- \frac{\Gamma(b)\Gamma(1-b\nu)}{\Gamma(b - b\nu)} \, F(\ell-b,-b\nu,-b;x) 
= b\nu(1-\nu) \int_0^1 t^{-b\nu-1}(1-t)^\ell(1-(1-x)t)^{b-\ell} \, 
\mathrm{d}t.
$$
Now, replacing the latter into the right-hand side of (\ref{F0}) provides
\begin{align}
\mathfrak{L}E^*(z) = & \, \frac{x\nu(1-\nu)}{1-x} \, 
\sum_{\ell \geqslant 1} 
\frac{(-1)^\ell}{\ell!} \, E_\ell \; \times
\nonumber \\
& \, \sum_{b \geqslant \ell} \frac{b}{(b-\ell)!} \cdot \left ( - \frac{z}{x} \right )^b \, \int_0^1 t^{-b\nu-1}(1-t)^\ell(1-(1-x)t)^{b-\ell} \, 
\mathrm{d}t
\label{F3}
\end{align}
where, with the index change $b' = b - \ell$, the latter sum over 
$b \geqslant \ell$ for given $\ell$ equivalently reads
\begin{align}
& \sum_{b \geqslant \ell} \frac{b}{(b-\ell)!} \cdot 
\left ( - \frac{z}{x} \right )^b \, 
\int_0^1 t^{-b\nu-1}(1-t)^\ell(1-(1-x)t)^{b-\ell} \, \mathrm{d}t \; =
\nonumber \\
& \sum_{b' \geqslant 0} \frac{b'+\ell}{b'!} \cdot 
\left ( - \frac{z}{x} \right )^{b'+\ell} \, 
\int_0^1 t^{-(b'+\ell)\nu}(1-t)^\ell(1-(1-x)t)^{b'} \, 
\frac{\mathrm{d}t}{t} \; = 
\nonumber \\
& z \, \frac{\mathrm{d}}{\mathrm{d}z} \left [ 
\sum_{b' \geqslant 0} \frac{1}{b'!} \cdot 
\left ( - \frac{z}{x} \right )^{b'+\ell} \, 
\int_0^1 t^{-(b'+\ell)\nu}(1-t)^\ell(1-(1-x)t)^{b'} \, 
\frac{\mathrm{d}t}{t} \right ] \; = 
\nonumber \\
& z \, \frac{\mathrm{d}}{\mathrm{d}z} 
\left [ \left ( - \frac{z}{x} \right )^{\ell} 
\int_0^1 t^{-\ell \nu} (1-t)^\ell \, \frac{\mathrm{d}t}{t} \times 
\exp \left ( - \frac{z}{x} \, t^{-\nu}[1-(1-x)t] \right ) \right ].
\label{F4}
\end{align}
Replacing expression (\ref{F4}) into the left-hand side of (\ref{F3}), the linearity of operator $\delta = z \, \mathrm{d}/\mathrm{d}z$ and the permutation of the summation on index $\ell$ with the integration with respect to variable $t \in [0,1]$ enable us to obtain
$$
\mathfrak{L}E^*(z) = \frac{x\nu(1-\nu)}{1-x} \cdot  
\delta \Biggl [ \int_0^1 \frac{\mathrm{d}t}{t} \, 
e^{- \frac{z}{x} \, t^{-\nu}[1-(1-x)t]} \; 
\sum_{\ell \geqslant 1} 
\frac{E_\ell}{\ell!} 
\left ( \frac{z}{x} \right )^{\ell} t^{-\ell \nu}(1-t)^\ell \Biggr ],
$$
that is,
\begin{align}
\mathfrak{L}E^*(z) = & \, \frac{x\nu(1-\nu)}{1-x} \cdot  
\delta \Biggl [ \int_0^1 \frac{\mathrm{d}t}{t} 
e^{- \frac{z}{x} \, t^{-\nu}[1-(1-x)t]} \times 
E^* \left ( \frac{z}{x} \, t^{-\nu}(1-t) \right ) \Biggr ] 
\nonumber \\
= & \, \frac{x\nu(1-\nu)}{1-x} \cdot (\delta \circ \mathfrak{M})E^*(z), 
\qquad z \in \mathbb{C},
\nonumber
\end{align}
for any function $E^* \in \mathscr{H}_0$, as claimed in (\ref{DefTTbis}), with the corresponding definition of integral operator $\mathfrak{M}$ on space 
$\mathscr{H}_0$. 
\end{proof}

As mentioned in the Introduction, the factorization (\ref{DefTTbis}) of operator $\mathfrak{L}$ allows one to write equation (\ref{EI0}) equivalently as 
\begin{equation}
\mathfrak{M}E^* = K_1
\label{EI0bis}
\end{equation}
where the given $K_1 \in \mathscr{H}_0$ relates to the initial function $K$ as introduced in (\ref{DefK1}). As function $\tau:t \in [0,1] \mapsto t^{-\nu}(1-t)$ has a unique maximum at point $\widehat{t} = \nu/(\nu - 1)$ for $\nu < 0$, we can introduce the variable changes 
$t \in [0,\widehat{t}] \mapsto \tau_-(t) = t^{-\nu}(1-t)$ and 
$t \in [\widehat{t},1] \mapsto \tau_+(t) = t^{-\nu}(1-t)$ on segments 
$[0,\widehat{t}]$ and $[\widehat{t},1]$, respectively; we further denote by
\begin{equation}
\theta_- = \tau_-^{-1}, \qquad \theta_+ = \tau_+^{-1}
\label{Thetapm}
\end{equation}
the respective inverse mappings of $\tau_-$ and $\tau_+$, both defined on segment $[0,\tau*]$ where $\widehat{\tau} = 
\tau(\widehat{t}) = (\widehat{t})^{-\nu}(1-\widehat{t})$ (see illustration on Fig.\ref{GraphTau}). The variable changes $\tau_-$ and $\tau_+$ then allow us to write equation (\ref{EI0bis}) as a singular Volterra equation.

\begin{figure}[htb]
\scalebox{1}{\includegraphics[width=13cm, trim = 0cm 10.5cm 11cm 1.5cm,clip]
{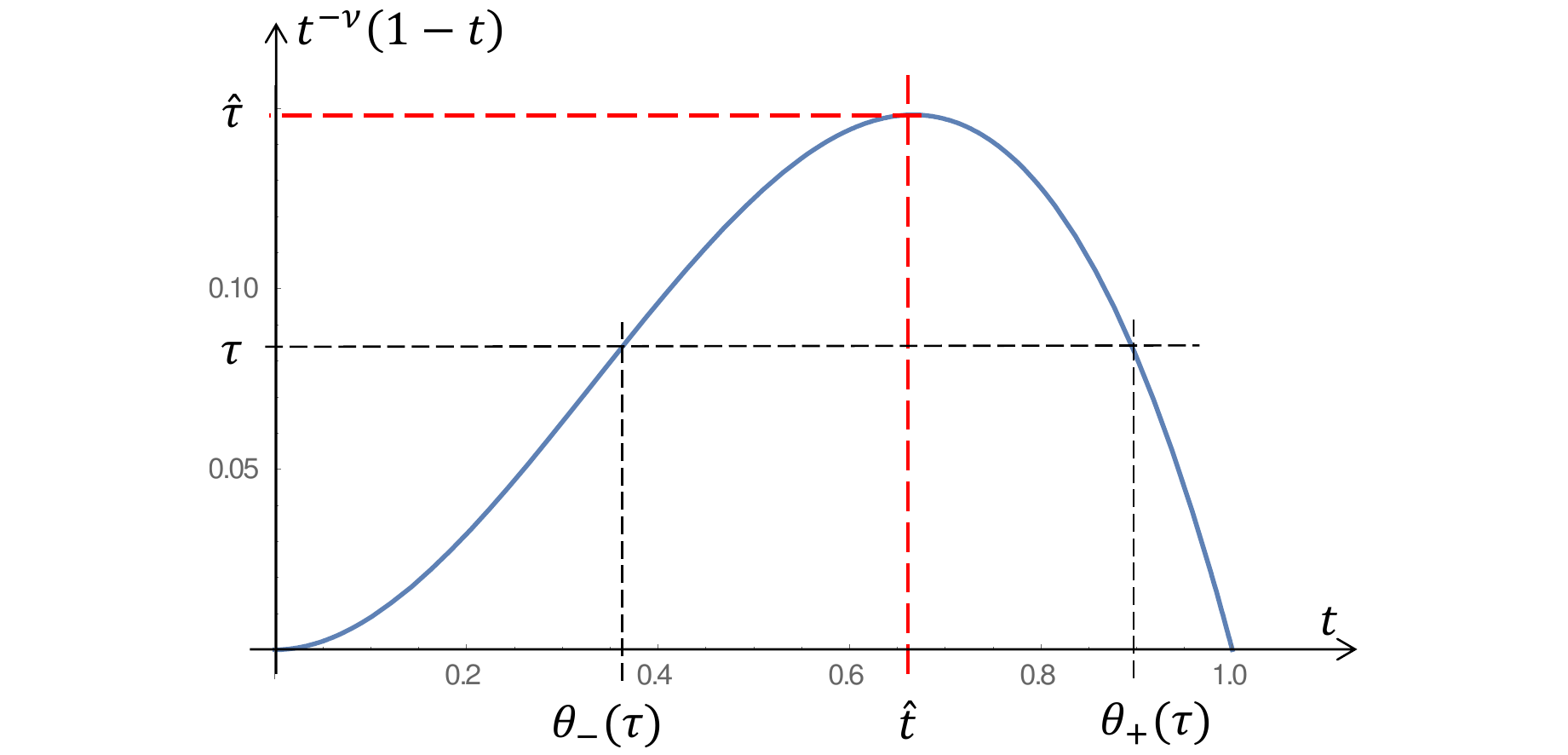}}
\caption{\textit{Graph of function $\tau:t \in [0,1] \mapsto t^{-\nu}(1-t)$ 
($\nu = -2$).}}
\label{GraphTau}
\end{figure}

\begin{corol}
\textbf{Given constants $x$ and $\nu$ as above, the equivalent equation 
(\ref{EI0bis}) can be recast into the singular Volterra integral equation}
\begin{equation}
\int_0^{\left ( \frac{\widehat{\tau}}{x} \right ) \, z} 
\left [ \Psi_- \left ( z, \frac{x\xi}{z}\right ) - 
\Psi_+ \left ( z, \frac{x\xi}{z} \right ) \right ] 
E^*(\xi) \, \mathrm{d}\xi = 
\frac{z}{x} \cdot K_1(z), \qquad z \in \mathbb{C},
\label{VOL0}
\end{equation}
\textbf{where we set}
$$
\Psi_{\pm}(z,\tau) = 
\frac{e^{- \frac{z}{x} 
\theta_{\pm}(\tau)^{-\nu}(1-(1-x)\theta_{\pm}(\tau))}}
{\theta_{\pm}(\tau)^{-\nu}(-\nu+(\nu-1)\theta_{\pm}(\tau))}, 
\qquad 0 \leqslant \tau \leqslant \widehat{\tau},
$$
\textbf{with $\theta_{\pm}$ introduced in (\ref{Thetapm}), and where 
$K_1 \in \mathscr{H}_0$ is defined by (\ref{DefK1}).}
\label{C3}
\end{corol}

\noindent
We refer to Appendix \ref{A5} for the proof of Corollary \ref{C3}. 

%%%%%%%%%%%%%%%%%%%%%%%%%
\begin{remark}
%%%%%%%%%%%%%%%%%%%%%%%%%
A few observations can be brought at this stage:

$\bullet$ as detailed in Appendix \ref{A5}, the kernel 
$\tau \mapsto \Psi_-(z,\tau) - \Psi_+(z,\tau)$ of Volterra equation 
(\ref{VOL0}) is singular with an integrable singularity at the boundary 
$\tau = \widehat{\tau}$ of order $O(\widehat{\tau} - \tau)^{-1/2}$;

$\bullet$ although giving a remarkable formulation to initial equation 
(\ref{EI0}), equation (\ref{VOL0}) is nevertheless difficult to solve  directly as its kernel depends on inverse functions $\theta_-$ and $\theta_+$ which cannot be made explicit simply (note that implicit equation 
$\tau = t^{-\nu}(1-t)$ in $t$ can be formulated as equation 
(\ref{DefTheta}) considered above, with $w = -\tau$ and $\nu$ replaced by 
$1 + \nu$, for which the analytic solution $\theta_-$ can be locally expressed near the origin via power series $\mathbf{\Sigma}$). 
%%%%%%%%%%%%%%%%%%%%%%%%%
\end{remark}
%%%%%%%%%%%%%%%%%%%%%%%%%

We now provide an integral representation for the inverse of operator 
$\mathfrak{L}$ on $\mathscr{H}_0$ or, equivalently, an integral representation for the solution of integral equation (\ref{EI0}) initially addressed in the Introduction.

\begin{corol}
\textbf{Let $\nu \in \mathbb{C}$ with $\mathrm{Re}(\nu) < 0$. Then}
\begin{itemize}
\item[\textbf{a)}] \textbf{the operator $\mathfrak{L}:\mathscr{H}_0 \rightarrow \mathscr{H}_0$ is a bijection;}
\item[\textbf{b)}] \textbf{given $K \in \mathscr{H}_0$, the solution 
$E^* = \mathfrak{L}^{-1}K \in \mathscr{H}_0$ to the integral equation 
(\ref{EI0}) has the integral representation}
\begin{align}
E^*(z) & = \, \mathfrak{L}^{-1} \, K(z) 
\nonumber \\
& = \, \frac{1-x}{2i\pi x} \, e^{z} 
\int_{(0+)}^1 \frac{e^{-xtz}}{t(1-t)} \, 
K \left ( x \, z \, (-t)^{\nu}(1-t)^{1-\nu} \right ) \, \mathrm{d}t,
\quad z \in \mathbb{C},
\label{SolEI0bis}
\end{align}
\end{itemize}
\textbf{where the contour in integral (\ref{SolEI0}) in variable $t$ is a 
loop starting and ending at point $t = 1$, and encircling the origin 
$t = 0$ once in the positive sense.}
\label{C4}
\end{corol}

\begin{proof} 
\textbf{a)} Given $K \in \mathscr{H}_0$, equation (\ref{EI0}) for 
$E^* \in \mathscr{H}_0$ is equivalent to system (\ref{T0}) for the coefficients $(E_\ell)_{\ell \geqslant 1}$ of the exponential series expansion of $E^*$. For $\mathrm{Re}(\nu) < 0$, Corollary \ref{C2} entails these coefficients are uniquely determined by expression 
(\ref{E0}). The linear operator $\mathfrak{L}:\mathscr{H}_0 \rightarrow \mathscr{H}_0$ is consequently one-to-one and onto, and has an inverse 
$\mathfrak{L}^{-1}$ on $\mathscr{H}_0$.

\textbf{b)} An integral representation for the inverse operator 
$\mathfrak{L}^{-1}$ is now derived as follows. Setting $S = E$ and 
$T = \widetilde{K}$ in (\ref{GsExp}), with the sequence 
$\widetilde{K} = (\widetilde{K}_b)_{b \geqslant 1}$ defined as in 
(\ref{S1}), we obtain
\begin{align}
\mathfrak{G}_{E}^*(z) & = e^z \cdot 
\sum_{b \geqslant 1} (-1)^b \widetilde{K}_b \, \frac{z^b}{b!} \, 
\Phi(b\nu;b;-x \, z)
\nonumber \\
& = - \frac{1-x}{x} \, e^z \cdot 
\sum_{b \geqslant 1} (-1)^b \frac{\Gamma(b - b \nu)}
{\Gamma(b)\Gamma(1-b \nu)} \cdot K_b \, \frac{(xz)^b}{b!} \, 
\Phi(b\nu;b;-x \, z)
\label{Sol0}
\end{align}
after using (\ref{S1}) to express $\widetilde{K}_b$ in terms of $K_b$, 
$b \geqslant 1$. Invoke then the integral representation
\begin{equation}
\Phi(\alpha;\beta;Z) = - \frac{1}{2 i \pi}
\frac{\Gamma(1-\alpha)\Gamma(\beta)}{\Gamma(\beta-\alpha)} 
\int_1^{(0+)} e^{Zt}(-t)^{\alpha - 1}(1-t)^{\beta-\alpha-1} \, 
\mathrm{d}t
\label{IRK1bis}
\end{equation}
of the Confluent Hypergeometric function $\Phi(\alpha;\beta;\cdot)$ for 
$\mathrm{Re}(\beta - \alpha) > 0$ (see \cite{ERD81}, Sect.6.11.1, (3))), where the integration contour is specified as in the Corollary (see 
Fig.\ref{IC}, red solid line). On account of (\ref{IRK1bis}) applied to 
$\alpha = b\nu$ and  $\beta = b \in \mathbb{N}^*$ with 
$\mathrm{Re}(\nu) < 1$, expression (\ref{Sol0}) now reads
\begin{align}
\mathfrak{G}^*_{E}(z) & = \frac{1-x}{2i\pi \, x} \, e^{z} 
\int_1^{(0)^+} \frac{e^{-xzt} \, \mathrm{d}t}{t(t-1)} 
\sum_{b \geqslant 1} (-1)^b \frac{K_b}{b!} \, 
(xz(-t)^{\nu}(1-t)^{1-\nu})^b
\nonumber \\
& = \frac{1-x}{2i\pi x} \, e^{z} 
\int_1^{(0+)} \frac{e^{-xtz}}{t(t-1)} \, 
K \left ( x \, z \, (-t)^{\nu}(1-t)^{1-\nu} \right ) \, \mathrm{d}t
\label{Sol1bis}
\end{align}
for all $z \in \mathbb{C}$. As $E^*(z) = \mathfrak{G}_E^*(z)$ by definition, expression (\ref{Sol1bis}) readily yields the final representation (\ref{SolEI0bis}), as claimed.
\end{proof}

\begin{figure}[htb]
\scalebox{0.8}{\includegraphics[width=14cm, trim = 4cm 7.5cm 3cm 
4.1cm,clip]
{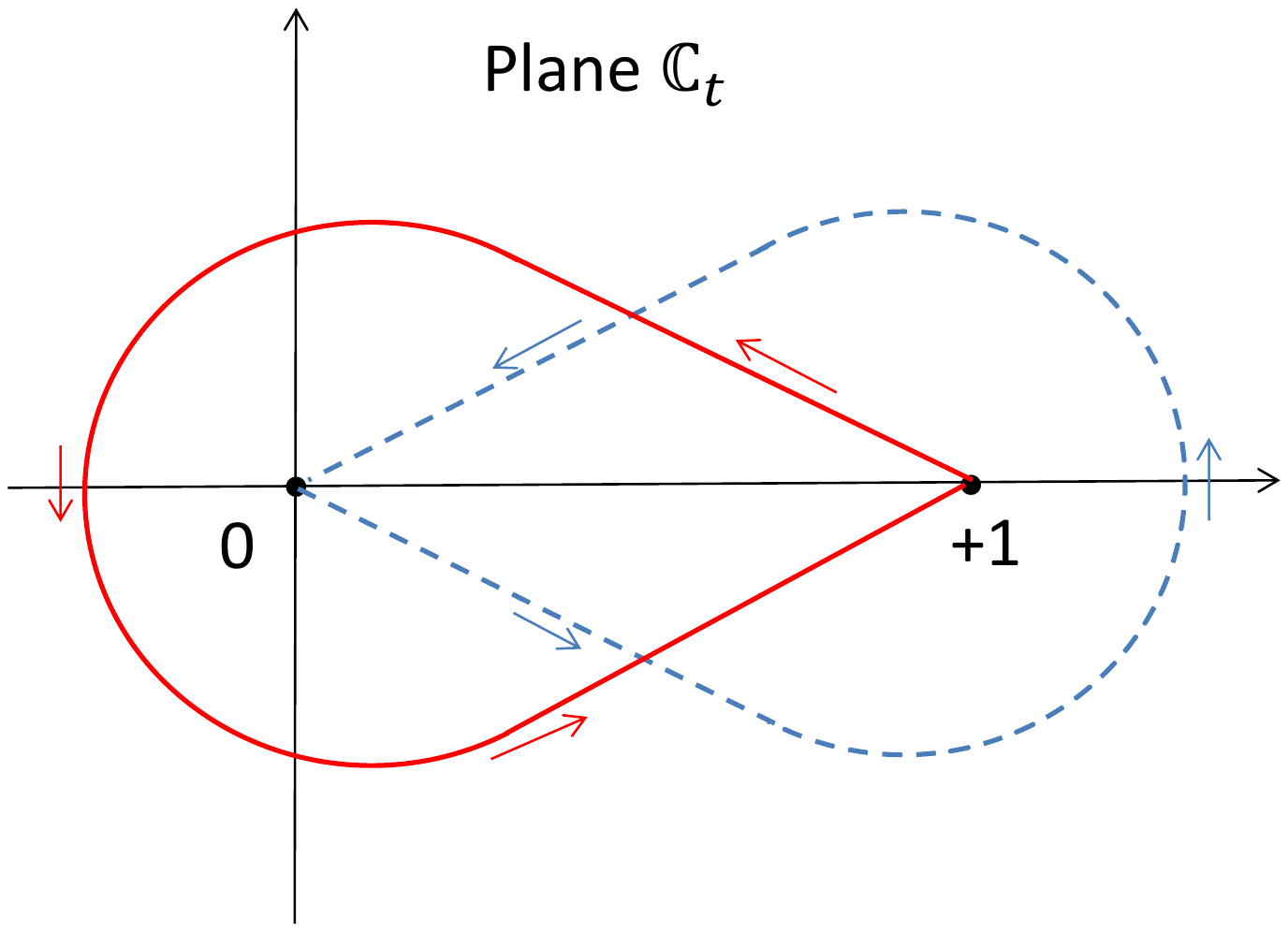}}
\caption{\textit{Integration contours around points 0 and 1.}}
\label{IC}
\end{figure}

By the factorization (\ref{DefTTbis}), it is consequently deduced that the inverse of integral operator $\mathfrak{M}$ is given by
$$
\mathfrak{M}^{-1}f(z) = \frac{x \, \nu(1-\nu)U}{1-x} \cdot 
\mathfrak{L}^{-1}(z \, f')(z), \qquad z \in \mathbb{C},
$$
for all $f \in \mathscr{H}_0$, with inverse $\mathfrak{L}^{-1}$ provided by integral representation (\ref{SolEI0bis}). The involvement of the derivative $f'$ for the inverse $\mathfrak{M}^{-1}f$ above reminds us of formula (\ref{SolAbel}) in the particular case of the Abel equation. 

\begin{remark}
$\bullet$ As mentioned in the latter proof, representation 
(\ref{SolEI0bis}) for the inverse $\mathfrak{L}^{-1}$ is actually valid for $\mathrm{Re}(\nu) < 1$, although the operator $\mathfrak{L}$ is defined on space $\mathscr{H}_0$ for $\mathrm{Re}(\nu) < 0$ only. By the variable change $t \mapsto 1-t$, representation (\ref{SolEI0bis}) can be easily written as
\begin{equation}
\mathfrak{L}^{-1}K(z) = \frac{x-1}{2i\pi x} \, e^{(1-x)z} 
\int_0^{(1+)} \frac{e^{xtz}}{t(1-t)} \, 
K \left ( xz \, t^{1-\nu}(t-1)^\nu \right ) \, \mathrm{d}t,
\quad z \in \mathbb{C},
\label{SolEI0}
\end{equation}
where the contour in integral (\ref{SolEI0}) in variable $t$ is a 
loop starting and ending at point $t = 0$, and encircling point $t = 1$ once in the positive sense (see Fig.\ref{IC}, blue dotted line).  Alternative representation (\ref{SolEI0}) generally holds, however, for 
$\mathrm{Re}(\nu) < 0$ only (with no extension to positive values of 
$\mathrm{Re}(\nu)$).

$\bullet$ To illustrate the fact that the operator $\mathfrak{L}$, although well-defined for $\mathrm{Re}(\nu) < 0$, may not exist for other values of parameter $\nu$, consider the particular function 
$f \in \mathscr{H}_0$ defined by $f(z) = z \, e^{(1-x)z}$, 
$z \in \mathbb{C}$. Using definition (\ref{DefTT}), it is easily verified that, for $\mathrm{Re}(\nu) < 0$, its image 
$\mathfrak{L}f \in \mathscr{H}_0$ is given by
$$
\mathfrak{L}f(z) = - \, \frac{z}{1-x} \, 
\Phi \left ( 1 - \frac{1}{\nu};1 - \frac{1}{\nu}; -z \right ), 
\qquad z \in \mathbb{C}
$$
(where $\Phi(\alpha;\beta;\cdot)$ denotes the Confluent Hypergeometric function with parameters $\alpha, \, \beta \notin \mathbb{N}$); for 
$\mathrm{Re}(\nu) > 0$, however, its image is given by
$$
\mathfrak{L}f(z) = - \, \frac{1-\nu}{\nu(1-x)} \, 
z^{\frac{1}{\nu}} \, \Gamma \left ( 1 - \frac{1}{\nu}; z \right ), 
\qquad z \neq 0
$$
(where $\Gamma(\cdot;\cdot)$ denotes the incomplete Gamma function), so that $\mathfrak{L}f \notin \mathscr{H}_0$ in this case.
\end{remark}

\section{Appendix}

%%%%%%%%%%%%%%%%%%%%%%%%%%%%%%%%%
%%%%%%%%%%%%%%%%%%%%%%%%%%%%%%%%%

%%%%%%%%%%%%%%%%%%%%%%%%%%%%%%%%%%%%%%%%%%%%%
\subsection{Proof of Lemma \ref{lemm1}}
\label{A1}
%%%%%%%%%%%%%%%%%%%%%%%%%%%%%%%%%%%%%%%%%%%%%
\textbf{a)} By the reflection formula 
$\Gamma(z)\Gamma(1-z) = \pi/\sin(\pi \, z)$, $z \notin -\mathbb{N}$ 
(\cite{NIST10}, §5.5.3), applied to the argument $z = r-\mu$, the generic term $d_r(\lambda,\mu)$ of the sum $D_N(\lambda,\mu)$ equivalently reads
$$
d_r(\lambda,\mu) = \frac{(-1)^r}{\Gamma(1+r-\lambda)\Gamma(1-r+\mu)} = 
- \frac{\sin(\pi \mu)}{\pi} \, \frac{\Gamma(r-\mu)}{\Gamma(1+r-\lambda)}
$$
and Stirling's formula (\cite{NIST10}, §5.11.3) entails that 
$d_r(\lambda,\mu) = O(r^{\lambda-\mu-1})$ for large $r$; the series 
$\sum_{r \geqslant 0} d_r(\lambda,\mu)$ is thus  convergent if and only if 
$\Re(\mu) > \Re(\lambda)$. Write then the finite sum $D_N(\lambda,\mu)$ as the difference
\begin{align}
& \, \sum_{r=0}^{+\infty} \frac{(-1)^r}{\Gamma(1+r-\lambda)\Gamma(1-r+\mu)} - 
\sum_{r=N}^{+\infty}\frac{(-1)^r}{\Gamma(1+r-\lambda)\Gamma(1-r+\mu)} \; = 
\nonumber \\
& \, \sum_{r=0}^{+\infty}\frac{(-1)^r}{\Gamma(1+r-\lambda)\Gamma(1-r+\mu)} - 
\sum_{r=0}^{+\infty}\frac{(-1)^{r+N}}{\Gamma(1+r+N-\lambda)\Gamma(1-r-N+\mu)};
\nonumber
\end{align}
applying similarly the reflection formula to the argument $z = r-\mu+N$ for the second sum, we obtain
\begin{align}
D_N(\lambda,\mu) & \, = \frac{\sin(\pi \, \mu)}{\pi} 
\left [ \sum_{r=0}^{+\infty} \frac{\Gamma(r-\mu+N)}{\Gamma(1+r+N-\lambda)} 
- \sum_{r=0}^{+\infty} \frac{\Gamma(r-\mu)}{\Gamma(1+r-\lambda)} \right ]
\nonumber \\
& \, = \frac{\sin(\pi \, \mu)}{\pi} 
\left [ \sum_{r=0}^{+\infty} 
\frac{(N-\mu)_r\Gamma(r-\mu)}{(1+N-\lambda)_r\Gamma(1+N-\lambda)} 
- \sum_{r=0}^{+\infty} \frac{(-\mu)_r\Gamma(-\mu)}{(1-\lambda)_r\Gamma(1-\lambda)} \right ]
\nonumber
\end{align}
when introducing Pochhammer symbols of order $r$, hence
\begin{align}
D_N(\lambda,\mu) = \frac{\sin(\pi \, \mu)}{\pi} 
\Bigl [ & \, \frac{\Gamma(N-\mu)}{\Gamma(1+N-\lambda)} \, F(1,N-\mu;1+N-\lambda;1) \; - 
\nonumber \\
& \, \frac{\Gamma(-\mu)}{\Gamma(1-\lambda)} \, F(1,-\mu;1-\lambda;1) \Bigr ]
\nonumber
\end{align}
after the definition of the Hypergeometric function $F$. Now, recall the identity (\cite{GRAD07}, §9.122.1)
\begin{equation}
F(\alpha,\beta;\gamma;1) = \frac{\Gamma(\gamma)\Gamma(\gamma-\alpha-\beta)}
{\Gamma(\gamma-\alpha)\Gamma(\gamma-\beta)}, \qquad 
\Re(\gamma) > \Re(\alpha + \beta);
\label{HyperF1}
\end{equation}
when aplying (\ref{HyperF1}) to the values $\alpha = 1$, $\beta = N-\mu$, 
$\gamma = 1 + N -\lambda$ (resp. $\alpha = 1$, $\beta = -\mu$, 
$\gamma = 1 - \lambda$), the latter sum $D_N(\lambda,\mu)$ consequently reduces to 
\begin{equation}
D_N(\lambda,\mu) = \frac{\sin(\pi \, \mu)}{\pi} 
\frac{\Gamma(\mu-\lambda)}{\Gamma(1-\lambda+\mu)} 
\left [ \frac{\Gamma(N-\mu)}{\Gamma(N-\lambda)} - \frac{\Gamma(-\mu)}{\Gamma(-\lambda)} \right ], \quad \Re(\mu) > \Re(\lambda).
\label{A11}
\end{equation}
By the reflection formula for function $\Gamma$ again, we have
$$
\Gamma(N-\mu)\Gamma(1-N+\mu) = - \frac{(-1)^N \pi}{\sin(\pi \mu)}, \qquad 
\Gamma(-\mu)\Gamma(1+\mu) = - \frac{\pi}{\sin(\pi \mu)},
$$
so that expression (\ref{A11}) eventually yields 
\begin{align}
D_N(\lambda,\mu) & \, = - \frac{\Gamma(\mu-\lambda)}{\Gamma(1-\lambda+\mu)} 
\left [ \frac{(-1)^N}{\Gamma(N-\lambda)\Gamma(1-N+\mu)} - 
\frac{1}{\Gamma(-\lambda)\Gamma(1+\mu)} \right ] 
\nonumber \\
& \, = \frac{1}{\lambda-\mu} 
\left [ \frac{(-1)^N}{\Gamma(N-\lambda)\Gamma(1-N+\mu)} - 
\frac{1}{\Gamma(-\lambda)\Gamma(1+\mu)} \right ]
\nonumber
\end{align}
which states the first identity (\ref{Sn}) for $\Re(\mu) > \Re(\lambda)$.

\textbf{b)} Besides, the reflection formula of function $\Gamma$ applied to 
$z = r-\lambda$ enables us to write $D_N(\lambda,\lambda)$ as
\begin{align}
D_N(\lambda,\lambda) & \, = 
\sum_{r=0}^{N-1} \frac{(-1)^r}{\Gamma(1+r-\lambda)\Gamma(1-r+\lambda)} 
= - \frac{\sin(\pi \lambda)}{\pi} \sum_{r=0}^{N-1} 
\frac{\Gamma(r-\lambda)}{\Gamma(1-r+\lambda)} 
\nonumber \\
& \, = - \frac{\sin(\pi \lambda)}{\pi} \sum_{r=0}^{N-1} \frac{1}{r-\lambda} 
= \frac{\sin(\pi \lambda)}{\pi} 
\left [ \psi(-\lambda) - \psi(N-\lambda) \right ]
\nonumber
\end{align}
after the expansion formula (\cite{NIST10}, Chap.5, §5.7.6) for function 
$\psi$ and the second identity (\ref{Sn}) for $\mu = \lambda$ follows.

\textbf{c)} The first identity (\ref{Sn}) stated for 
$\Re(\mu) > \Re(\lambda)$ defines an analytic function of variables 
$\lambda \in \mathbb{C}$ and $\mu \in \mathbb{C}$ for $\mu \neq \lambda$; besides, it is easily verified that this function has the limit given by 
$D_N(\lambda,\lambda)$ when $\mu \rightarrow \lambda$. On the other hand, the finite sum $D_N(\lambda,\mu)$ defines itself an entire function of 
$\lambda \in \mathbb{C}$ and $\mu \in \mathbb{C}$; by analytic continuation, identity (\ref{Sn}) consequently holds for any pair 
$(\lambda,\mu) \in \mathbb{C} \times \mathbb{C}$ $\blacksquare$

%%%%%%%%%%%%%%%%%%%%%%%%%%%%%%%%%%%%%%%%%%%%%%%%%%%%%%%%%%%%%%%
\subsection{Proof of Lemma \ref{lemmU}}
\label{A3}
%%%%%%%%%%%%%%%%%%%%%%%%%%%%%%%%%%%%%%%%%%%%%%%%%%%%%%%%%%%%%%%
\textbf{a)} We first determine the convergence radius of the power series 
$\pmb{\Sigma}(w)$ in terms of complex parameter $\nu$. For large $b$, 

$\bullet$ if $1 - \nu \notin \; ]-\infty,0]$ and $-\nu \notin \; ]-\infty,0]$, that is, if $\nu \in \mathbb{C} \setminus [0,+\infty[$, the generic term 
$\sigma_b$ of this series is asymptotic to
$$
\sigma_b = \frac{\Gamma(b(1-\nu))}{\Gamma(b)\Gamma(1-b\nu)} 
= - \frac{1}{\nu} \cdot \frac{\Gamma(b(1-\nu))}{b! \, \Gamma(-b\nu)} 
\sim - \sqrt{\frac{-\nu}{2\pi(1-\nu)b}} \, e^{b \cdot \varphi^-(\nu)}
$$
after Stirling's formula $\Gamma(z) \sim \sqrt{2\pi} e^{z \log z -z}/\sqrt{z}$ for large $z$ with $\vert \mathrm{arg}(z) \vert \leqslant \pi - \eta$, 
$\eta > 0$ (\cite{NIST10}, Chap.5, 5.11.3), and where we set  
$\varphi^-(\nu) = (1-\nu) \log (1-\nu) + \nu \log(-\nu)$;

$\bullet$ if $1 - \nu \notin \; ]-\infty,0]$ and $\nu \in [0,+\infty[$ (the parameter $\nu$ is consequently real), that is, $0 \leqslant \nu < 1$, write 
$\Gamma(1-b\nu) = \pi / [\sin(\pi b \nu) \Gamma(b\nu)]$ after the reflection formula so that the generic term $\sigma_b$ is now asymptotic to
$$
\sigma_b = - \frac{1}{\nu} \cdot \frac{\Gamma(b(1-\nu))}{b! \, \pi} 
\Gamma(b\nu) \, \sin(\pi b \nu) \sim 
- \frac{1}{\nu} \sqrt{\frac{\pi}{2\nu(1-\nu)b^3}} \, \sin(\pi b \nu) 
\, e^{b \cdot \varphi(\nu)}
$$
after Stirling's formula (ibid.) and where
$\varphi(\nu) = (1-\nu) \log (1-\nu) + \nu \log(\nu)$; 

$\bullet$ finally if $\nu - 1 \in [0,+\infty]$, that is, if $\nu \geqslant 1$, write $\Gamma(1-b\nu) = \pi / [\sin(\pi b \nu) \Gamma(b\nu)]$ together with  
$\Gamma(1-b(1-\nu)) = \pi / [\sin(\pi b (1-\nu)) \Gamma(b(1-\nu))]$ after the reflection formula so that the generic term $\sigma_b$ is asymptotic to
$$
\sigma_b = \frac{(-1)^{b-1}}{\nu} \cdot 
\frac{\Gamma(b\nu)}{b! \, \Gamma(1-b(1-\nu))} \sim 
\frac{(-1)^{b-1}}{\nu} \sqrt{\frac{1}{2\pi \, \nu(\nu - 1)b^3}} \, 
e^{b \cdot \varphi^+(\nu)}
$$
after Stirling's formula and where
$\varphi^+(\nu) = (1-\nu) \log (\nu-1) + \nu \log(\nu)$. 

\textbf{b)} By the latter discussion, it therefore follows that the power series 
$\pmb{\Sigma}(w)$ has the finite convergence radius 
$R(\nu) = \vert e^{-\psi(\nu)} \vert$ with $\psi(\nu)$ 
%defined by 
%$$
%\psi(\nu) = \left\{
%\begin{array}{ll}
%\varphi^-(\nu), \quad \; \, 
%\nu \in \mathbb{C} \setminus \, [0,+\infty[,
%\\ \\
%\varphi(\nu), \; \quad \; \; \, \, 
%\nu \in \mathbb{R}, \; 0 \leqslant \nu < 1, 
%\\ \\
%\varphi^+(\nu), \; \quad \; 
%\nu \in \mathbb{R}, \; \nu \geqslant 1,
%\end{array} \right.
%$$
given as in Lemma \ref{lemmU}.

Now, by the above expression of $\sigma_b$ for 
$\nu \in \mathbb{C} \setminus [0,+\infty[$, write
\begin{equation}
\sigma_b = - \frac{1}{\nu} \cdot \frac{\Gamma(b(1-\nu))}{b! \, \Gamma(-b\nu)}  = - \frac{1}{\nu} \cdot \binom{-1 + b(1-\nu)}{b} = - \frac{1}{\nu} \cdot 
\binom{\alpha + b\beta}{b}
\label{DefSig0}
\end{equation}
for all $b \geqslant 1$, where we set $\alpha = -1$ and $\beta = 1-\nu$. From 
(\cite{PolSze72}, Problem 216, p.146, p. 349), it is known that 
\begin{equation}
1 + \sum_{b \geqslant 1} \binom{\alpha + b\beta}{b} w^b = 
\frac{\Theta(w)^{\alpha+1}}{(1-\beta)\Theta(w) + \beta}
\label{PolSzeg0}
\end{equation}
for any pair $\alpha$ and $\beta$, where $\Theta(w)$ denotes the unique solution to the implicit equation $1 - \Theta + w\, \Theta^\beta = 0$
with $\Theta(0) = 1$. By expression (\ref{DefSig0}) and relation 
(\ref{PolSzeg0}) applied to the specific values $\alpha = -1$ and 
$\beta = 1-\nu$, we can consequently assert that the series $\pmb{\Sigma}(w)$ equals
$$
\pmb{\Sigma}(w) = \sum_{b \geqslant 1} \sigma_b \, w^b = - \frac{1}{\nu} 
\left [ \frac{1}{\nu \, \Theta(w) + 1-\nu} - 1 \right ] = 
\frac{\Theta(w)-1}{\nu \, \Theta(w) + 1-\nu}
$$
for $\vert w \vert < R(\nu)$, as claimed. The validity of equality 
(\ref{U0}) for real $\nu \in [0,+\infty[$ follows by analytic continuation. 
%$\blacksquare$

%%%%%%%%%%%%%%%%%%%%%%%%%%%%%%%%%%%
%\begin{comment}
%%%%%%%%%%%%%%%%%%%%%%%%%%%%%%%%%%%
\textbf{c)} As a complement, we finally verify that $\pmb{\Sigma}$ is a solution to differential equation (\ref{EDiff0}). Differentiating each side of the implicit relation (\ref{DefTheta}) at point $w \neq 0$ gives
$$
-\Theta'(w) + \Theta(w)^{1-\nu} + w \, \Theta'(w) \Theta(w)^{-\nu}(1-\nu) = 0
$$
hence 
\begin{align}
\Theta'(w) = & \, \frac{\Theta(w) ^{1-\nu}}{1 - w \, \Theta(w)^{-\nu}(1-\nu)} 
= \frac{(\Theta(w)-1)/w}{1 - w(\Theta(w)-1)(1-\nu)/w\Theta(w)} 
\nonumber \\
= & \, \frac{\Theta(w)}{w} \, \frac{\Theta(w)-1}{\nu \, \Theta(w) + 1 - \nu}
\nonumber 
\end{align}
after using relation (\ref{DefTheta}) again for $\Theta(w)^{1-\nu}$; using  relation (\ref{U0}), the latter expression for $\Theta'(w)$ consequently reduces to 
\begin{equation}
\Theta'(w) = \frac{\Theta(w)}{w} \, \pmb{\Sigma}(w).
\label{Theta1}
\end{equation}
Now, differentiating (\ref{U0}) at point $w$ and using (\ref{Theta1}) yields
\begin{equation}
\pmb{\Sigma}'(w) = \frac{\Theta'(w)}{(\nu \, \Theta(w) + 1 - \nu)^2} = 
\frac{\Theta(w) \, \pmb{\Sigma}(w)}{w(\nu \, \Theta(w) + 1 - \nu)^2};
\label{Theta2}
\end{equation}
but solving (\ref{U0}) for $\Theta(w)$ in terms of $\pmb{\Sigma}(w)$ readily gives the rational expressions
$$
\Theta(w) = \frac{(1-\nu)\pmb{\Sigma}(w) + 1}{1-\nu \, \pmb{\Sigma}(w)}, 
\qquad \nu \Theta(w) + 1 - \nu = \frac{1}{1 - \nu \pmb{\Sigma}(w)}
$$ 
which, once replaced into the right-hand side of (\ref{Theta2}), entail
$$
\pmb{\Sigma}'(w) = 
\frac{\displaystyle \frac{(1-\nu)\pmb{\Sigma}(w)+1}{1-\nu \, \pmb{\Sigma}(w)} \times 
\pmb{\Sigma}(w)}{w \displaystyle \left ( \frac{1}{1 - \nu \pmb{\Sigma}(w)} \right )^2}
$$
and readily provide differential equation (\ref{EDiff0}) after algebraic reduction $\blacksquare$
%%%%%%%%%%%%%%%%%%%%%%%%%%%%%%%%%%
%\end{comment}
%%%%%%%%%%%%%%%%%%%%%%%%%%%%%%%%%%

%%%%%%%%%%%%%%%%%%%%%%%%%%%%%%%%%%%%%%%%%%%%%%%%%%%%%%%%%%%%%%%
\subsection{Proof of Proposition \ref{LemT0} (continued)}
\label{A4}
%%%%%%%%%%%%%%%%%%%%%%%%%%%%%%%%%%%%%%%%%%%%%%%%%%%%%%%%%%%%%%%
We conclude the proof of Proposition \ref{LemT0} by expressing the coefficients $Q_{b,\ell}$, $1 \leqslant \ell \leqslant b$, introduced in (\ref{DefQbell}) in terms of Hypergeometric polynomials only. We first calculate coefficients $Q_{b,\ell}(s)$, $1 \leqslant \ell \leqslant b$, in terms of the general Gauss Hypergeometric function $F$. Recall that 
$F = F(\alpha,\beta;\gamma;\cdot)$ has the integral representation (\cite{NIST10}, Chap.15, 15.6.1)
\begin{equation}
F(\alpha,\beta;\gamma;z) = 
\frac{\Gamma(\gamma)}{\Gamma(\beta)\Gamma(\gamma-\beta)}
\int_0^1 \frac{t^{\beta-1}(1-t)^{\gamma-\beta-1}}{(1-z t)^\alpha} 
\, \mathrm{d}t, \quad \vert z \vert < 1,
\label{HyperGauss}
\end{equation}
for real parameters $\alpha$, $\beta$, $\gamma$ where 
$\gamma > \beta > 0$.

\begin{lemma}
\textbf{We have}
\begin{align}
Q_{b,\ell} = & \, 
- \frac{\Gamma(\ell)\Gamma(1 - b \, \nu)}{\Gamma(\ell + 1 - b \, \nu)} 
\, \left ( \frac{x}{1-x}  \right ) \, 
\Bigl [  \, \nu \, (b-\ell) \; \times 
\nonumber \\
& \,F(b \, (1-\nu),\ell;\ell + 1 - b \, \nu;1-x)  + (\ell - b \, \nu) \times 
F(b \, (1-\nu),\ell;\ell - b \, \nu;1-x) 
\, \Bigr ]
\label{CoeffQ}
\end{align}
\textbf{for $1 \leqslant \ell \leqslant b$.}
\label{CoeffT0}
\end{lemma}

\begin{proof}
To calculate the integral $M_{b,\ell}$ introduced in (\ref{DefMbell}), use the variable change $\zeta = t \cdot U$, 
$0 \leqslant t \leqslant 1$, to write
$$
M_{b,\ell} = \int_0^1 
t^\ell (1-t)^{-b \, \nu} 
\left ( 1 - (1-x)t \right )^{b \, (\nu - 1)} \, 
\mathrm{d}t;
$$
using representation (\ref{HyperGauss}) for parameters 
$\alpha = -b(1-\nu)$, $\beta = \ell + 1$, $\gamma = 2 + \ell - b \, \nu$, this integral reduces to
\begin{equation}
M_{b,\ell} = 
\frac{\Gamma(\ell+1)\Gamma(1- b \, \nu)}{\Gamma(2+\ell -  b \, \nu)} \, 
F(b \, (1-\nu),\ell+1;2+\ell-b \, \nu;1-x);
\label{MBL0}
\end{equation}
after (\ref{MBL0}) and the expression (\ref{DefQbell}) of coefficient 
$Q_{b,\ell}$, we then derive
\begin{align}
Q_{b,\ell} = & \, 
\frac{\Gamma(\ell)\Gamma(1 - b \, \nu)}{\Gamma(\ell + 1 - b \, \nu)} 
\, \times 
\nonumber \\
& \, \Bigl [  \frac{\ell}{\ell+1-b \, \nu} (\ell+1-b) \cdot 
F(b \, (1-\nu),\ell+1;\ell + 2 - b \, \nu;1-x) 
\nonumber \\
& \; - \ell \, c \cdot F(b \, (1-\nu),\ell;\ell + 1 - b \, \nu;1-x) 
\, \Bigr ].
\label{MBL1}
\end{align}
To simplify further the latter expression, first invoke the identity 
\begin{equation}
\beta \, F(\alpha,\beta+1;\gamma+1;z) = \gamma \, 
F(\alpha,\beta;\gamma;z) - 
(\gamma-\beta) \, F(\alpha,\beta;\gamma+1;z)
\label{ID1}
\end{equation}
easily derived from representation (\ref{HyperGauss}) for 
$F(\alpha,\beta+1;\gamma+1;z)$, after splitting the factor $t^\beta$ of the integrand into $t^\beta = t^{\beta-1} - t^{\beta-1}(1-t)$. Applying 
(\ref{ID1}) to $\alpha = b \, (1-\nu)$, $\beta = \ell$ and 
$\gamma = \ell + 1 - b \, \nu$ then enables one to express the term 
$F(b \, (1-\nu),\ell+1;\ell + 2 - b \, \nu;1-x)$ in the r.h.s. of (\ref{MBL1}) as a combination of 
$F(b \, (1-\nu),\ell;\ell + 1 - b \, \nu;1-x)$ and 
$F(b \, (1-\nu),\ell;\ell + 2 - b \, \nu;1-x)$ hence, after simple algebra,
\begin{align}
Q_{b,\ell} = & \, 
\frac{\Gamma(\ell)\Gamma(1 - b \, \nu)}{\Gamma(\ell + 1 - b \, \nu)} 
\, \Bigl [ \left\{ (\ell + 1 - b) - \ell \, c \right \} \cdot 
F(b \, (1-\nu),\ell;\ell + 1 - b \,\nu;1-x) 
\nonumber \\
& \, - \frac{(\ell+1-b)(1-b \, \nu)}{\ell + 1 - b \, \nu} \cdot 
F(b \, (1-\nu),\ell;\ell + 2 - b \, \nu;1-x) \, \Bigr ].
\label{MBL2}
\end{align}
Furthermore, the contiguity identity (\cite{NIST10}, 15.5.18) 
\begin{align}
& \gamma[\gamma-1-(2\gamma-\alpha-\beta-1)z] \, F(\alpha,\beta;\gamma;z) 
\; + 
\nonumber \\
& (\gamma-\alpha)(\gamma-\beta)z \, F(\alpha,\beta;\gamma+1;z) = 
\gamma(\gamma-1)(1-z) \, F(\alpha,\beta;\gamma-1;z)
\end{align}
applied to $\alpha = b \, (1-\nu)$, $\beta = \ell$ and 
$\gamma = \ell + 1 - b \, \nu$ allows us to write the last term 
$F(b \, (1-\nu),\ell;\ell + 2 - b \, \nu;1-x)$ in the bracket of the r.h.s. of (\ref{MBL2}) as a combination of $F(b \, (1-\nu),\ell;\ell - b \, \nu;1-x)$ and $F(b \, (1-\nu),\ell;\ell + 1 - b \, \nu;1-x)$, that is,
\begin{align}
& F(b \, (1-\nu),\ell;\ell + 2 - b \, \nu;1-x) = 
\frac{\ell + 1 - b \, \nu}{(\ell + 1 - b)(1-b \, \nu)(1-x)} \; \times 
\nonumber \\
& \Bigl [ (\ell - b \, \nu)x \cdot 
F(b \, (1-\nu),\ell;\ell - b \, \nu;1-x) \; - 
\nonumber \\
& [\ell - b \, \nu \; - (\ell + 1 - b - b \, \nu)(1-x)] \cdot 
F(b \, (1-\nu),\ell;\ell + 1 - b \, \nu;1-x) \Bigr ];
\nonumber
\end{align}
inserting the latter relation into the right-hand side of (\ref{MBL2}) then yields
\begin{align}
Q_{b,\ell} = & \, 
\frac{\Gamma(\ell)\Gamma(1 - b \, \nu)}{\Gamma(\ell + 1 - b \, \nu)} 
\, \times \Bigl [ T_{b,\ell} \cdot 
F(b \, (1-\nu),\ell;\ell + 1 - b \, \nu;1-x) \; -  
\nonumber \\
& \, \frac{U}{1-x}(\ell - b \, \nu)x \cdot 
F(b \, (1-\nu),\ell;\ell - b \, \nu;1-x) \, \Bigr ]
\label{MBL3}
\end{align}
where 
$$
T_{b,\ell} = b \, \nu - \ell \, c + \frac{(\ell - b \, \nu)}{1-x} = 
\frac{(\ell - b) \nu x}{1-x}
$$
after the definition $c = (1-\nu x)/(1-x)$ of constant $c$. Inserting this value of $T_{b,\ell}$ in the right-hand side of (\ref{MBL3}) readily provides expression (\ref{CoeffQ}) for $Q_{b,\ell}$. 
\end{proof}

We finally show how coefficient $Q_{b,\ell}$ can be written in terms of a Hypergeometric polynomial only. Applying the general identity (\cite{GRAD07}, Chap.9, 9.131.1)
\begin{equation}
F(\alpha,\beta;\gamma;z) = (1-z)^{\gamma-\alpha-\beta}
F(\gamma-\alpha,\gamma-\beta;\gamma;z), \qquad \vert z \vert < 1,
\label{GI0}
\end{equation}
to each term $F(b(1-\nu),\ell;\ell+1-b\nu;1-x)$ and 
$F(b(1-\nu),\ell;\ell-b\nu;1-x)$ in (\ref{CoeffQ}), we obtain
\begin{equation}
Q_{b,\ell} = 
- \frac{\Gamma(\ell)\Gamma(1 - b \, \nu)}{\Gamma(\ell + 1 - b \, \nu)} 
\cdot \frac{x^2}{b (1-x)} (\ell - b\nu) \, \times 
R_{b,\ell}, 
\qquad b \geqslant \ell \geqslant 1, 
\label{Q1}
\end{equation}
where we set
\begin{align} 
R_{b,\ell} = & \; \frac{b\nu \, (b-\ell)}{\ell - b \nu} x^{-b} \cdot 
F(\ell - b + 1,-b\nu + 1;\ell - b\nu + 1;1-x) \; + 
\nonumber \\
& \; b \, x^{-b-1} \cdot F(\ell-b,-b\nu;\ell-b\nu;1-x).
\nonumber
\end{align}
From the identity (\cite{NIST10}, Chap.15, §15.5.1) 
\begin{equation}
\frac{\mathrm{d}}{\mathrm{d}z} \, F(\alpha,\beta;\gamma;z) = 
\frac{\alpha \beta}{\gamma} \, 
F(\alpha+1,\beta+1;\gamma+1;z), \qquad \vert z \vert < 1,
\label{DF}
\end{equation}
applied to parameters $\alpha = \ell - b$, $\beta = -b\nu$ and 
$\gamma= \ell - b\nu$, the factor $R_{b,\ell}$ above then equals the derivative 
\begin{align}
R_{b,\ell} = & \; \frac{\mathrm{d}}{\mathrm{d}z} \left [ (1-z)^{-b}F(\ell-b,-b\nu;\ell-b\nu;z) \right ]_{z = 1-x} 
\nonumber \\ = & \; \frac{\mathrm{d}}{\mathrm{d}z} 
\left [ F(b-b\nu,\ell;\ell-b\nu;z) \right ]_{z = 1-x}
\nonumber \\
= & \; \frac{(b-b\nu)\ell}{\ell-b\nu} \, 
F(b-b\nu+1,\ell+1;\ell-b\nu+1;1-x)
\nonumber
\end{align}
hence
\begin{equation}
R_{b,\ell} = \frac{(b-b\nu)\ell}{\ell-b\nu} \, x^{-b-1}
F(\ell - b,-b\nu;\ell-b\nu+1;1-x)
\label{Q1bis}
\end{equation}
where we have successively applied identity (\ref{GI0}), (\ref{DF}) and 
(\ref{GI0}) again to derive the second, third and fourth equality, respectively. Using (\ref{Q1bis}), expression (\ref{Q1}) for 
$Q_{b,\ell}$ then reads
\begin{equation}
Q_{b,\ell} = - \frac{\Gamma(\ell)\Gamma(1-b\nu)}{\Gamma(\ell-b\nu)} \, 
\frac{ x^{1-b}}{1-x} \; 
\frac{\ell(1-\nu)}{\ell - b\nu} \; S_{b,\ell}(\nu;1-x), 
\; \; b \geqslant \ell \geqslant 1,
\label{Q3}
\end{equation}
where we set
$$
S_{b,\ell}(\nu;1-x) = F(\ell-b,-b\nu;\ell-b\nu+1;1-x)
$$
for short. To reduce further $S_{b,\ell}(\nu;1-x)$, invoke the identity 
(\cite{NIST10}, Chap.15, §15.8.7)
\begin{equation}
F(-m,\beta,\gamma;1-x) = 
\frac{\Gamma(\gamma)\Gamma(\gamma-\beta+m)}
{\Gamma(\gamma-\beta)\Gamma(\gamma+m)} \, F(-m,\beta,\beta+1-m-\gamma;x), 
\quad x \in \mathbb{C},
\label{IDFG}
\end{equation}
for any non negative integer $m$ and complex numbers $\beta$, $\gamma$ such that $\Re(\gamma) > \Re(\beta)$; applying (\ref{IDFG}) to factor 
$S_{b,\ell}(\nu;1-x)$ in (\ref{Q3}) then readily gives the final expression 
(\ref{Q0}) for all indexes $b \geqslant \ell \geqslant 1$. This concludes the proof of Proposition \ref{LemT0} $\blacksquare$

%%%%%%%%%%%%%%%%%%%%%%%%%%%%%%%%%%%%%%%%%%%%%%%%%%%%%%%%%%%%%%%
\subsection{Proof of Corollary \ref{C3}}
\label{A5}
%%%%%%%%%%%%%%%%%%%%%%%%%%%%%%%%%%%%%%%%%%%%%%%%%%%%%%%%%%%%%%%
$\bullet$ From the definition (\ref{DefSS}) of integral operator 
$\mathfrak{M}$, split the integral 
\begin{align}
\mathfrak{M}E^*(z) = & \, 
\int_0^1 e^{- \frac{z}{x} \, t^{-\nu}(1-(1-x)t)} \, 
E^* \left ( \frac{z}{x} t^{-\nu}(1-t) \right ) \frac{\mathrm{d}t}{t} 
\nonumber \\
= & \, \int_0^{\widehat{t}} \; (...) \, \frac{\mathrm{d}t}{t} + 
\int_{\widehat{t}}^1 \; (...) \, \frac{\mathrm{d}t}{t}
\nonumber
\end{align}
over adjacent segments $[0,\widehat{t}]$ and $[\widehat{t},1]$, respectively; applying the variable change $\tau = t^{-\nu}(1-t)$ on each of these two intervals with 
$\tau = \tau_-(t) \Leftrightarrow t = 
\theta_-(\tau) \in [0,\widehat{t}]$ and 
$\tau = \tau_+(t) \Leftrightarrow t = 
\theta_+(\tau) \in [\widehat{t},1]$ by the definition (\ref{Thetapm}) of mappings $\theta_-$ and $\theta_+$, we then successively obtain
\begin{align}
\mathfrak{M}E^*(z) = & \, \int_0^{\widehat{\tau}} 
e^{- \frac{z}{x} \theta_-(\tau)^{-\nu}(1-(1-x)\theta_-(\tau))} 
E^* \left ( \frac{z}{x} \, \tau \right ) \, 
\frac{-\mathrm{d}\tau}{\theta_-(\tau)^{-\nu}(\nu+(1-\nu)\theta_-(\tau))} 
\nonumber \\
+ & \, \int_{\widehat{\tau}}^0
e^{- \frac{z}{x} \theta_+(\tau)^{-\nu}(1-(1-x)\theta_+(\tau))} 
E^* \left ( \frac{z}{x} \, \tau \right ) \, 
\frac{-\mathrm{d}\tau}{\theta_+(\tau)^{-\nu}(\nu+(1-\nu)\theta_+(\tau))}
\nonumber
\end{align}
with $\widehat{\tau} = \tau_-(\widehat{t}) = \tau_+(\widehat{t})$ and the differential 
$\mathrm{d}t/t = -\mathrm{d}\tau/[t^{-\nu}(\nu+(1-\nu)t)]$; this readily reduces to a single integral over segment $[0,\widehat{\tau}]$, that is,
$$
\mathfrak{M}E^*(z) = \int_0^{\widehat{\tau}} 
\left [ \Psi_-(z,\tau) - \Psi_+(z,\tau) \right ]  
E^* \left ( \frac{z}{x} \, \tau \right ) \, \mathrm{d}\tau 
$$
with $\Psi_-(z,\tau)$ and $\Psi_+(z,\tau)$ given as in the Corollary. The final variable change $\xi = (z/x) \cdot \tau$ yields the right-hand side  of (\ref{VOL0}) and the corresponding integral equation. 

$\bullet$ We finally verify that the r.h.s. of (\ref{VOL0}) is well-defined for any $E^* \in \mathscr{H}_0$. The denominator 
$t^{-\nu}(-\nu + (\nu-1)t)$ of $\Psi_-(z,\tau)$ with 
$t = \theta_-(\tau)$ (resp. of $\Psi_+(z,\tau)$ with 
$t = \theta_+(\tau)$) vanishes at either $\tau = 0$ or 
$\tau = \widehat{\tau}$ (resp. at $\tau = \widehat{\tau}$). As to the possible singularity at $\tau = 0$ for $\Psi_-(z,\tau)$, we have 
$\tau \sim t^{-\nu}$ for small $t = \theta_-(\tau)$ so that 
$$
\frac{1}{t^{-\nu}(-\nu + (\nu-1)t)} \sim - \, \frac{t^{\nu}}{\nu} \sim 
- \frac{1}{\nu\tau}, \qquad \tau \downarrow 0;
$$
the product $E^*(z\tau/x) \cdot \Psi_-(z,\tau)$ is thus integrable near 
$\tau = 0$ for any $E^* \in \mathscr{H}_0$, as required. Besides, a Taylor expansion of $\tau = \tau(t)$ at order 2 near $t = \widehat{t}$ gives
$$
\tau = \widehat{\tau} + 
\frac{\tau''(\widehat{t})}{2} \, (t-\widehat{t})^2 
+ o(t-\widehat{t})^2
$$
with $\tau'(\widehat{t}) = 0$ by definition and 
$\tau''(\widehat{t}) < 0$; as a result,
$$
t - \widehat{t} \sim 
\pm \sqrt{\frac{-2(\widehat{\tau}-\tau)}{\tau''(\widehat{t})}}, 
\qquad \tau \uparrow \widehat{\tau}.
$$
The denominator $t^{-\nu}(-\nu+(\nu-1)t)$ of either $\Psi_-(z,\tau)$ or 
$\Psi_+(z,\tau)$ is consequently asymptotic to
$$
t^{-\nu}(-\nu+(\nu-1)t) \sim 
(\widehat{t})^{-\nu}(\nu-1)(t - \widehat{t}) \sim 
\pm (\widehat{t})^{-\nu}(\nu-1) 
\sqrt{\frac{2(\widehat{\tau}-\tau)}{-\tau''(\widehat{t})}}
$$
when $\tau \uparrow \widehat{\tau}$; the singularity of 
$\Psi_-(z,\tau)$ (resp. $\Psi_+(z,\tau)$) at point $\tau = \widehat{\tau}$ is consequently of order 
$$
\Psi_-(z,\tau) = O \left ( \frac{1}{\sqrt{\widehat{\tau}-\tau}} \right ), 
\quad 
\Psi_+(z,\tau) = O \left ( \frac{1}{\sqrt{\widehat{\tau}-\tau}} \right )
$$
and the kernel $\Psi(z,\cdot) = \Psi_-'(z,\cdot) - \Psi_+(z,\cdot)$ is thus integrable at $\tau = \widehat{\tau}$. This ensures that the singular integral (\ref{VOL0}) is well-defined for any $E^* \in \mathscr{H}_0$ $\blacksquare$

%%%%%%%%%%%%%%%%%%%%%%%%%%%%%%%%%%%%%%%
%%%%%%%%%%%%%%%%%%%%%%%%%%%%%%%%%%%%%%%


\begin{thebibliography}{13}
%%%%%%%%%%%%%%%%%%%%%%%%%%%%%%%%%%%%

\bibitem{BIT95} A.V. Bitsadze, \textit{Integral Equations of First Kind}, ed. World Scientific, 1995

\bibitem{EST00} R. Estrada, R.P. Kanwal, \textit{Singular Integral Equations}, ed. Birhäuser, 2000

\bibitem{ERD81} A. Erdelyi, \textit{Higher Transcendental Functions}, Vol.1, 
ed. MacGraw Hill, 1981

%\bibitem{GouldHsu73} H.W. Gould, L.C. Hsu, \textit{Some new inverse series relations}, Duke Math. Journal, 40, N°4, pp. 885-891, 1973

\bibitem{GOR91} R. Gorenflo, S. Vessella, \textit{Abel Integral Equations, Analysis and Applications}, Lecture Notes in Mathematics 1461, ed. Springer, 1991

\bibitem{GRAD07} I.S. Gradsteyn, I.M. Ryzhik, \textit{Table of Integrals, Series and Products}, ed. Academic Press, 2007

\bibitem{GQSN18} F. Guillemin, V.K. Quintuna Rodriguez, A. Simonian, R. Nasri, \textit{Sojourn time in a $M^{[X]}/M/1$ Processor Sharing Queue with Batch Arrivals (II)}, In Preparation, 2018

\bibitem{Kratten96} C. Krattenthaler, \textit{A new Matrix Inverse}, Proceedings of the American Mathematical Society 124, pp.47-59, 1996

\bibitem{KRESS14} R. Kress, \textit{Linear Integral Equations}, Third edition, ed. Springer 2014

\bibitem{NIST10} National Institute of Standards and Technology, 
\textit{NIST Handbook of Mathematical Functions}, ed. Cambridge University Press, 2010

\bibitem{PolSze72} G. Polya, G. Szego, \textit{Problems and Theorems in Analysis}, Vol.I, ed. Springer, 1972

\bibitem{PolMan98} A.D. Polyanin, A.V. Manzhirov, \textit{Handbook of Integral equations}, ed. CRC Press, 1998

\bibitem{Schlo97} M. Schlosser, \textit{Multidimensional Matrix Inversions and $A_r$ and $D_r$ Basic Hypergemeotric series}, The Ramanujan Journal, I, 
pp.243-274, 1997

%%%%%%%%%%%%%%%%%%%%%%%%%%%%%%%%%%%%%
\end{thebibliography}
\end{document}